\documentclass[]{article}
\usepackage[scale=0.7]{geometry}
\usepackage{amssymb,amsthm,amsmath,geometry,bm,hyperref}
\usepackage{algorithm}
\usepackage[noend]{algpseudocode}
\usepackage{color,xcolor}
\usepackage{graphicx,subcaption}
\usepackage{booktabs,multicol,multirow}
\usepackage{comment}

\makeatletter

\newcommand{\Rmnum}[1]{\expandafter\@slowromancap\romannumeral #1@}

\makeatother
\usepackage{authblk}
\usepackage{hyperref}                                                  
\hypersetup{hypertex = true, colorlinks = true, linkcolor = blue, anchorcolor = blue, citecolor =blue}
\usepackage[square,numbers]{natbib}
\newtheorem{theorem}{Theorem}
\newtheorem{lemma}{Lemma}

\newtheorem{definition}{Definition}

\newtheorem{remark}{Remark}

\newcommand{\vecl}{\boldsymbol{\ell}}
\newcommand{\vech}{\boldsymbol{h}}
\newcommand{\veck}{\boldsymbol{k}}

\newcommand{\vecx}{\boldsymbol{x}}
\newcommand{\vecb}{\boldsymbol{b}}
\newcommand{\vecy}{\boldsymbol{y}}
\newcommand{\vecgamma}{\boldsymbol{\gamma}}

\newcommand{\setH}{\mathcal{H}}
\newcommand{\vecz}{\boldsymbol{0}}
\newcommand{\vecg}{\boldsymbol{g}}
\newcommand{\setA}{A_{\alpha,\vecgamma,M}}

\newcommand{\set}[1]{\left\{#1\right\}}
\newcommand{\seq}[1]{\left(#1\right)}

\newcommand{\R}{\mathbb{R}}
\newcommand{\Z}{\mathbb{Z}}

\newcommand{\C}{\mathbb{C}}
\newcommand{\bz}{\boldsymbol{z}}
\newcommand{\bDelta}{\boldsymbol{\Delta}}

\newcommand{\SpaceKor}{\mathcal{K}_{d, \alpha, \vecgamma}}

\title{A deterministic multiple-shift lattice algorithm for function approximation in Korobov and half-period Cosine spaces}

\author[1]{Jiarui Du}
\author[2]{Josef Dick \thanks{Corresponding author: josef.dick@unsw.edu.au}}

\affil[1]{School of Mathematics, South China University of Technology, Guangzhou 510641, Guangdong, People’s Republic of China}
\affil[2]{School of Mathematics and Statistics, The University of New South Wales, Kensington, NSW 2052, Australia}
\begin{document}
	\maketitle

    \begin{abstract}
    Approximating multivariate periodic functions in weighted Korobov spaces via rank-1 lattices is fundamentally limited by frequency aliasing. Existing optimal-rate methods rely on randomized constructions or large pre-computations. We propose a fully deterministic multiple-shift lattice algorithm without pre-computation. First, we develop a simplified multiple shift framework for aliased frequency fibers that reduces sampling costs. Second, leveraging the Chinese Remainder Theorem and the Weil bound, we introduce an adaptive hybrid construction that algebraically guarantees the full rank and bounded condition number of the reconstruction matrix. We rigorously prove that this deterministic method maintains the optimal convergence rate in the worst-case setting.

    Furthermore, we extend this framework to non-periodic, half-period cosine spaces via the tent transformation. By establishing a strict projection equivalence, we prove that the algorithm attains optimal $L_2$ and $L_\infty$ approximation orders in the half-period cosine space, successfully resolving an open theoretical problem posed by Suryanarayana et al. (2016). This mathematically also validates the proposed algorithm as a generic meshless spectral solver for high-dimensional boundary value problems, such as the Poisson equation with Neumann conditions. Numerical experiments corroborate the theoretical bounds, demonstrating an order-of-magnitude reduction in sampling complexity over probabilistic baselines while ensuring absolute deterministic stability.
\end{abstract}

\maketitle

\section{Introduction}\label{sec:introduction}
The approximation of multivariate periodic functions in weighted Korobov spaces based on quasi-Monte Carlo methods is a fundamental problem, particularly for mitigating the curse of dimensionality in high-dimensional settings \cite{dick2022}. A classical approach employs rank-1 lattice rules, which was systematically extended from numerical integration to worst-case multivariate function approximation in \cite{kuo2004lattice}. The approach described therein is based on approximating the Fourier coefficients of the function from a certain finite set of frequencies. However, the exact reconstruction of these Fourier coefficients via deterministic single rank-1 lattices is inherently hindered by the aliasing effect, where high-frequency and low-frequency components become indistinguishable on the lattice nodes \cite{kammerer2013,kuo2021function}. In their seminal work, Byrenheid et al. \cite{byrenheid2017tight} established a theoretical barrier for this phenomenon, proving that any deterministic algorithm based on a single rank-1 lattice is subject to a strict lower bound of $\mathcal{O}(N^{-\alpha/2})$ for the worst-case $L_2$-approximation error, where $N$ is the lattice size and $\alpha > 1/2$ is the smoothness parameter.

To circumvent this theoretical lower bound and approach the optimal convergence rate of almost $\mathcal{O}(N^{-\alpha})$, research initially focused on structural decoupling. K\"ammerer and Volkmer \cite{Kammerer2019} proposed a reconstructing scheme via multiple rank-1 lattices to physically separate aliased frequencies, attaining optimal $L_\infty$ convergence in unweighted spaces. This framework leverages reconstruction theory for the stable recovery of hyperbolic-cross trigonometric polynomials from rank-1 samples \cite{kammerer2013}, alongside single-lattice approximation bounds \cite{kammerer2015approximation}. The algorithmic foundations for exact recovery and sparse FFT-type methods were established in \cite{kammerer2018}. This methodology was subsequently generalized to weighted Korobov spaces in the monograph by \cite{dick2022}, albeit under the conditions of smoothness $\alpha > 1$ and product weights. Cai and Goda \cite{cai2026note} extended the multiple rank-1 lattice framework to lower smoothness cases ($1/2 < \alpha \le 1$) and general weights, attaining nearly optimal randomized $L_2$ convergence by incorporating random shifts. 

Recent advancements heavily utilize randomization to refine these bounds and relax parameter restrictions. Cai et al. \cite{cai2024l_2} demonstrated that a randomized CBC algorithm combined with a random number of points can surpass the deterministic $\mathcal{O}(N^{-\alpha/2})$ barrier in expectation. Subsequently, Cai et al. \cite{cai2025} proposed applying $\mathcal{O}((\log N)^d)$ independent random shifts to a single rank-1 lattice. By solving a localized least-squares problem, this multiple-shift algorithm achieves optimal convergence rates in both the worst-case $L_\infty$ and randomized $L_2$ settings. Concurrently, median lattice algorithms \cite{pan2025universal,pan2025} achieve nearly optimal $L_2$ bounds with high probability via the median-of-means technique.

Despite these theoretical advancements, algorithms attaining the optimal $\mathcal{O}(N^{-\alpha+\epsilon})$ rate currently rely explicitly on probabilistic existence arguments or randomized constructions. In the deterministic regime, progress remains scarce. While Gross et al. \cite{gross2021} established a pioneering deterministic algebraic construction for the multiple-lattice framework, it intrinsically relies on the prior computation of a massive single base lattice of size $\mathcal{O}(|A|^2)$, where $|A|$ is the frequency set cardinality. Consequently, a deterministic construction capable of matching the optimal convergence rate without the heavy preprocessing bottleneck of a pre-computed large lattice remains absent from the literature.

Beyond the periodic setting, extending lattice-based approximation to non-periodic domains is critical for physical applications, such as high-dimensional partial differential equations (PDEs) \cite{li2003trigonometric, munthe2012multidimensional}. To overcome boundary singularities, a standard approach employs the tent transformation to map non-periodic functions into the periodic Korobov space \cite{cools2016tent, dick2014lattice}. While Suryanarayana et al. \cite{suryanarayana2016} established the numerical stability of collocation using tent-transformed lattices, the full error analysis intrinsically tied to the smoothness of the cosine space was explicitly left as an open problem. Furthermore, although recent frameworks unify discrete least-squares approximation across Fourier and cosine spaces \cite{kuo2021function}, a fully deterministic construction preserving optimal approximation orders in non-periodic settings remains unestablished.

In this paper, we bridge these theoretical gaps by proposing a fully deterministic multiple-shift lattice algorithm. By exploiting the algebraic properties of the Chinese Remainder Theorem (CRT) and the Weil bound, we develop a deterministic hybrid construction for the shift set that structurally guarantees the full rank of the reconstruction matrix and bounds its condition number, eliminating any dependence on probabilistic arguments or massive pre-computed lattices.

Our main contributions are: 
(i) We establish the first fully deterministic multiple-shift framework without heavy pre-computation for multivariate function approximation. 
(ii) We develop a uniform shift strategy combined with a deterministic hybrid construction. This fundamentally shifts the algebraic paradigm from exact de-aliasing—which suffers from combinatorial explosion—to bounded non-orthogonality, mathematically ensuring deterministic stability using a strictly logarithmically bounded shift count $S$. 
(iii) We rigorously prove that this deterministic construction tightens the worst-case error penalty term associated with the maximum fiber length $R$ from an implicit $\mathcal{O}(R^{3/2})$ to $\mathcal{O}(R^{1/2})$, maintaining the optimal convergence rate of $\mathcal{O}(N^{-\alpha+\epsilon})$. 
(iv) We extend this deterministic framework to the non-periodic setting via a tent transformation. By establishing a rigorous projection equivalence between the Korobov space and the half-period cosine space, we prove our algorithm attains optimal approximation orders for both $L_2$ and $L_\infty$ errors, resolving the open problem posed in \cite{suryanarayana2016}.

Consequently, we formulate the proposed algorithm as a generic meshless spectral solver for high-dimensional boundary value problems (e.g., the Poisson equation with Neumann boundary conditions). Comprehensive numerical experiments validate our theoretical bounds, demonstrating order-of-magnitude reductions in total sampling costs and guaranteed absolute stability compared to existing probabilistic shift strategies. MATLAB scripts (MATLAB R2023b) are available from \url{https://github.com/Jiarui-Du/multiple-shift}.
	
	\section{Preliminaries}
    \subsection{Weighted Korobov Space }\label{sec:kor}
    Let $d \ge 1$ be the dimension and $[0,1]^d$ be the domain. For a smoothness parameter $\alpha > 1/2$ and a sequence of positive weights $\vecgamma = (\gamma_1, \gamma_2, \dots)$, the weight function $r_{\alpha, \vecgamma}(\veck)$ for a frequency vector $\veck \in \Z^d$ is defined as
    $$
    r_{\alpha, \vecgamma}(\veck) = \prod_{j=1}^d \max\left\{ 1,\frac{|k_j|^\alpha}{\gamma_j}\right\}.
    $$
    \begin{definition}
        The weighted Korobov space $\SpaceKor$ is a reproducing kernel Hilbert space of one-periodic functions $f: \R^d \to \C$ with absolutely convergent Fourier series, defined by
    $$
    \SpaceKor := \set{f \in L_2([0,1]^d) : \|f\|^2_{\SpaceKor} :=  \sum_{\veck \in \Z^d} |\widehat{f}(\veck)|^2 r_{\alpha, \vecgamma}^2(\veck)  < \infty},
    $$
    where $\widehat{f}(\veck) = \int_{[0,1]^d} f(\vecx) e^{-2\pi \mathrm{i} \veck \cdot \vecx} d\vecx$.
    \end{definition}
    Our objective is to approximate Fourier coefficients for indices in a hyperbolic cross set defined as
    $$
    \setA := \{ \veck \in \Z^d : r_{\alpha, \vecgamma}(\veck) < M \}.
    $$
     Following \cite[Lemma 2.1]{cai2025}, we choose the parameter $M$ as
    \begin{equation}
        M = \sup_{1/\alpha < \lambda \le 2} \left( N\prod_{j=1}^d (1 + 2\gamma_j^\lambda \zeta(\alpha\lambda))^{-1} \right)^{1/\lambda}, \nonumber
    \end{equation}
    where $\zeta(\cdot)$ is the Riemann zeta function. This choice ensures that $|\setA| \le N$.

    \subsection{Rank-1 Lattices and Fibers}
    A rank-1 lattice point set with a prime number of points $N$ and a generating vector $\vecg \in \{1, \dots, N-1\}^d$ is defined as 
    $$\Lambda(\vecg, N) := \{ \{n\vecg/N\} : n = 0, \dots, N-1 \},$$ where $\{\cdot\}$ denotes the fractional part of each component. The associated dual lattice is given by
    \begin{equation}\label{eq:dual}
         \mathcal{L}_N^\perp(\vecg) := \{ \veck \in \Z^d : \veck \cdot \vecg \equiv 0 \pmod N \}.
    \end{equation}
    The orthogonality property of lattice rules implies that for any $\veck, \vecl \in \Z^d$, if $\vecl - \veck \in \mathcal{L}_N^\perp(\vecg)$, the functions $e^{2\pi \mathrm{i} \veck \cdot \vecx}$ and $e^{2\pi \mathrm{i} \vecl \cdot \vecx}$ are indistinguishable on $\Lambda(\vecg, N)$ 
    \begin{equation*}
    e^{2\pi \mathrm{i} \veck \cdot \{n \vecg /N\}} = e^{2\pi \mathrm{i} n \veck \cdot \vecg / N} = e^{2 \pi \mathrm{i} n (\veck + \vecl - \veck) \cdot \vecg / N} = e^{2\pi \mathrm{i} n \vecl \cdot \vecg / N} = e^{2\pi \mathrm{i} \vecl \cdot \{ n \vecg / N\}},
    \end{equation*}    
    leading to the well-known aliasing effect.
    
    For a given frequency $\vecl \in \setA$, we define the \textit{fiber} of $\vecl$ with respect to $\vecg$ as the set of all frequencies in the index set $\setA$ that alias with $\vecl$:
    $$
    \Gamma_{\alpha,\vecgamma,N}(\vecl; \vecg) := \{ \veck \in \setA : \veck \cdot \vecg \equiv \vecl \cdot \vecg \pmod N \}.
    $$
    Let $\vecl^{(1)}, \vecl^{(2)}, \dots, \vecl^{(J)} \in \setA$ be such that
    $$
    \Gamma_{\alpha,\vecgamma,N}(\vecl^{(i)}; \vecg) \cap \Gamma_{\alpha,\vecgamma,N}(\vecl^{(j)}; \vecg) = \emptyset,
    $$
    for $i \neq j$, and
    $$
    \bigcup_{j=1}^J \Gamma_{\alpha,\vecgamma,N}(\vecl^{(j)}; \vecg) = \setA.
    $$
    Let $R:= R_{\alpha,\vecgamma,N}(\vecg) := \max_{\vecl \in \setA} |\Gamma_{\alpha,\vecgamma,N}(\vecl; \vecg)|$ denote the maximum fiber length over $\setA$. This parameter $R$ is critical, as it determines the minimum number of shifts required for exact reconstruction in our proposed algorithm.

    \subsection{Construction of Generating Vector}
    To achieve computational efficiency and ensure theoretical convergence, we employ a deterministic component-by-component (CBC) construction for the generating vector $\vecg \in \{1, \dots, N-1\}^d$. This greedy algorithm constructs $\vecg$ by a component-by-component minimization of the squared worst-case error $P_{\alpha,\vecgamma,N}(\vecg)$ for numerical integration in $\SpaceKor$, defined as
    $$
    P_{\alpha,\vecgamma,N}(\vecg) = \sum_{\veck \in \mathcal{L}_N^\perp(\vecg) \setminus \{\vecz\}} \frac{1}{r_{\alpha,\vecgamma}^2(\veck)}.
    $$
    Minimizing this quantity effectively provides a lower bound for the figure of merit $$\rho(\vecg) := \min_{\veck \in \mathcal{L}_N^\perp(\vecg) \setminus \{\vecz\}} r_{\alpha,\vecgamma}(\veck).$$ For an integer smoothness parameter $\alpha \in \mathbb{N}$, $P_{\alpha,\vecgamma,N}(\vecg)$ possesses an explicit representation \cite{dick2022}:
    \begin{equation}\label{eq:worst_case_error}
        P_{\alpha,\vecgamma,N}(\vecg) = -1 + \frac{1}{N} \sum_{n=0}^{N-1} \prod_{j=1}^d \left( 1 + \gamma_j^2 \frac{(-1)^{\alpha+1} (2\pi)^{2\alpha}}{(2\alpha)!} B_{2\alpha}\left(\left\{ \frac{n g_j}{N} \right\}\right) \right),
    \end{equation}
    where $B_{2\alpha}(\cdot)$ denotes the Bernoulli polynomial of degree $2\alpha$. The construction procedure is detailed in Algorithm \ref{alg:CBCg}. As demonstrated in \cite{nuyens2006fast}, the CBC construction can be implemented using the Fast Fourier Transform (FFT) with $\mathcal{O}(dN \log N)$ arithmetic operations.

     \begin{algorithm}
    \caption{CBC Construction for $\vecg$}
    \label{alg:CBCg}
    \begin{algorithmic}[1]
    \Require Dimension $d$, smoothness $\alpha \in \mathbb{N}$, weights $\vecgamma$, prime $N$.
    \Ensure Generating vector $\vecg = (g_1, \dots, g_d)$.
    \State Set $g_1 = 1$.
    \For{$s = 1$ to $d-1$}
        \State Find $g_{s+1} \in \{1, \dots, N-1\}$ that minimizes $P_{\alpha,\vecgamma,N}((g_1, \dots, g_s, g_{s+1}))$.
    \EndFor
    \State \Return $\vecg = (g_1, \dots, g_d)$.
    \end{algorithmic}
    \end{algorithm}

    For non-integer $\alpha$, we utilize Jensen's inequality: for any $0 < \lambda \le 1$, 
    \begin{equation}
        P_{\alpha, \vecgamma, N}(\vecg) \le \left( P_{\alpha\lambda, \vecgamma^\lambda, N}(\vecg) \right)^{1/\lambda}.\nonumber
    \end{equation}
    To enable the use of \eqref{eq:worst_case_error}, we set $\lambda = \lfloor \alpha \rfloor / \alpha$ such that $\alpha\lambda \in \mathbb{N}$. We then execute Algorithm \ref{alg:CBCg} to minimize $P_{\lfloor \alpha \rfloor, \widetilde{\vecgamma}, N}(\vecg)$ with $\widetilde{\vecgamma} = \vecgamma^\lambda$.
    
    A critical property of this CBC construction is the guaranteed lower bound on the figure of merit:
    $$
    \rho(\vecg) \ge \sup_{1/\alpha < \lambda \le 2} \left( \frac{N}{2} \prod_{j=1}^d (1 + 2\gamma_j^\lambda \zeta(\alpha\lambda))^{-1} \right)^{1/\lambda}.
    $$
    By setting the truncation parameter $M$ to match this lower bound, we ensure that 
    \begin{equation}\label{eq_zero_fiber}
    \setA \cap \mathcal{L}_N^\perp(\vecg) = \{\vecz\}.
    \end{equation} 
    The following lemma, adapted from \cite[Lemma 3.4]{cai2025}, provides a bound on the maximum fiber length $R$.
    \begin{lemma}\label{lem:R}
        Let $N > 2$ be a prime number and $\gamma_j \in (0,1]$. Let $\vecg$ be constructed by Algorithm \ref{alg:CBCg} and $M$ be chosen as the lower bound of $\rho(\vecg)$.
        Let $m \in \Z$ satisfy $2^{m-1} < M^{1/\alpha} \le 2^m$. Then $|\setA| \le N/2$, $\rho(\vecg) \ge M$, and the maximum fiber length $R$ satisfies
        \begin{equation}
            R \le \min \left\{ 2 \frac{(1+m/2)^d}{m}, \quad 2 \frac{M^{1/\alpha}}{m} \prod_{j=1}^d (1 + \gamma_j^{1/\alpha} m) \right\} = \mathcal{O}((\log N)^{d-1}).
        \end{equation}
    \end{lemma}
    This poly-logarithmic bound is fundamental to our deterministic algorithm, ensuring that the required number of shifts (which scales with $R^2$) remains manageable even for high-dimensional problems.

    \begin{remark} \label{rem:trivial_d1}
    For $d = 1$, the CBC construction trivially yields the generator $g_1 = 1$. By our structural design, the truncation parameter $M$ is chosen such that the frequency set cardinality satisfies $|\setA| \le N$, which ensures that for any distinct $k_1, k_2 \in \setA$, their difference satisfies $|k_1 - k_2| < N$. Consequently, the aliasing condition $k_1 - k_2 \equiv 0 \pmod N$ is impossible, yielding a maximum fiber length of $R = 1$. Consequently, we focus on $d \ge 2$ and $R \ge 2$ throughout the remainder of this paper.
    \end{remark}

    \section{Review of the Multiple Shift Algorithm}
    We first review the multiple shift algorithm introduced in \cite{cai2025} for mitigating aliasing in Korobov spaces. This review highlights the computational challenges that motivate our simplified scheme.

    Consider a fiber $\Gamma_{\alpha,\vecgamma,N}(\vecl; \vecg) = \{\vecl_1, \dots, \vecl_v\}$ containing $v$ aliased frequencies. To distinguish these frequencies, the method in \cite{cai2025} assigns a specific set of shifts to each frequency within the fiber. For each frequency $\vecl_m$ ($m=1, \dots, v$) and each shift index $s=1, \dots, S$, a shift vector $\vecy_m^{(s)} \in [0,1]^d$ is employed. For $m=1, \dots, v$ and $s=1, \dots, S$, let
    \begin{equation} \label{eq:orig_observables}
        \begin{aligned}
        \mathcal{F}_N(f, \vecl_m, \vecy_m^{(s)}) &= \frac{1}{N} \sum_{n=0}^{N-1} f(\{n\vecg/N + \vecy_m^{(s)}\}) e^{-2\pi \mathrm{i} \vecl_m \cdot (n\vecg/N)} \\
        &= \frac{1}{N} \sum_{n=0}^{N-1} \sum_{\vech \in \Z^d}\widehat{f}(\vech)e^{2\pi \mathrm{i} \vech(n\vecg/N + \vecy_m^{(s)})} e^{-2\pi \mathrm{i} \vecl_m \cdot (n\vecg/N)}\\
        &=   \sum_{\vech \in \Z^d}\widehat{f}(\vech)e^{2\pi \mathrm{i} \vech\vecy_m^{(s)}} \frac{1}{N} \sum_{n=0}^{N-1}e^{2\pi \mathrm{i} (\vech-\vecl_m)n\vecg/N} \\
        &= \sum_{\vech \in \vecl_m+\mathcal{L}_N^\perp(\vecg)} \widehat{f}(\vech) e^{2\pi \mathrm{i} \vech \cdot \vecy_m^{(s)}} \approx \sum_{i=1}^v \widehat{f}(\vecl_i) e^{2\pi \mathrm{i} \vecl_i \cdot \vecy_m^{(s)}}.
        \end{aligned} 
    \end{equation} 
The estimates of the Fourier coefficients $\widehat{f}(\vecl_i)$ are obtained by solving the following least squares problem:
$$
\min_{\widehat{f}(\vecl_i)} \sum_{m=1}^v\sum_{s=1}^S \left(\sum_{i=1}^v \widehat{f}(\vecl_i) e^{2\pi \mathrm{i} \vecl_i \cdot \vecy_m^{(s)}} - \mathcal{F}_N\left(f, \vecl_m, \vecy_m^{(s)}\right)\right)^2.
$$
This yields the normal equations $B^{H}B\mathbf{x} = B^{H}\vecb$, where $\mathbf{x} = (\widehat{f}(\vecl_1), \dots, \widehat{f}(\vecl_v))^\top$, $\vecb$ contains the observations $\mathcal{F}_N$, and the system matrix $B \in \mathbb{C}^{vS \times v}$ is defined by its entries $B_{(m-1)S+s, i} = e^{2\pi \mathrm{i} \vecl_i \cdot \vecy_m^{(s)}}$. In \cite[Lemma 3.6]{cai2025}, it is shown that $\vecy_m^{(s)}$ can be chosen with high probability such that $B^{H}B$ is invertible.

We note that the least-squares approach remains valid when the shift vector 
$\vecy_m^{(s)}$ is chosen independently of $m$, that is, when a uniform shift 
$\vecy_s$ is used for all frequencies within a fiber. This simplifies the construction and, by avoiding frequency-dependent sampling locations, reduces the computational cost of estimating the Fourier coefficients.

\section{A Simplified Multiple Shift Algorithm}\label{sec:simp}
In this section, we propose a simplified approximation scheme for the frequencies $\vecl_1, \dots, \vecl_v \in \Gamma_{\alpha,\vecgamma,N}(\vecl; \vecg)$ by employing a uniform shift $\vecy_s$ for the entire fiber, rather than individual shifts for each frequency. This approach decouples the required number of function evaluations from the fiber length $v$. Furthermore, we introduce a global continuous shift $\boldsymbol{\Delta} \in [0,1]^d$ to provide a cohesive framework unifying both deterministic ($\boldsymbol{\Delta} = \vecz$) and randomized settings.

Let $\vecl_1$ be the representative frequency of the fiber $\Gamma_{\alpha,\vecgamma,N}(\vecl; \vecg) = \{\vecl_1, \dots, \vecl_v\}$. Since $\vecl_m \cdot \vecg \equiv \vecl_1 \cdot \vecg \pmod N$ for any $\vecl_m$ in the same fiber, the lattice rule sum $\mathcal{F}_N$ yields the same value for all frequencies within that fiber. For $s = 1, \dots, S$, we define the observations using the representative frequency:
$$
\begin{aligned}
b_s(\bDelta) &:= \mathcal{F}_N(f, \vecl_1, \vecy_s+\bDelta) = \frac{1}{N} \sum_{n=0}^{N-1} f(\{n\vecg/N + \vecy_s+\bDelta\}) e^{-2\pi \mathrm{i} \vecl_1 \cdot (n\vecg/N)} \\
&= \sum_{\veck \in \mathcal{L}_N^\perp(\vecg)} \widehat{f}(\veck + \vecl_1) e^{2\pi \mathrm{i} (\veck + \vecl_1) \cdot (\vecy_s+\bDelta)}.
\end{aligned}
$$
Truncating the sum to the fiber elements by setting $\veck + \vecl_1 = \vecl_m$, we obtain the approximation
$$
b_s(\bDelta) \approx \sum_{m=1}^v \widehat{f}(\vecl_m) e^{2\pi \mathrm{i} \vecl_m \cdot (\vecy_s+\bDelta)}.
$$
Consequently, we can relate the observations to the unknown coefficients via a diagonal scaling matrix $D(\bDelta)$:
$$
\mathbf{b}(\bDelta) \approx B D(\bDelta) \mathbf{x},
$$
where $\mathbf{x} = [\widehat{f}(\vecl_1), \dots, \widehat{f}(\vecl_v)]^\top$, $\mathbf{b}(\bDelta) = [b_1(\bDelta), \dots, b_S(\bDelta)]^\top$, and the system matrix $B \in \mathbb{C}^{S \times v}$ has entries $B_{s,m} = e^{2\pi \mathrm{i} \vecl_m \cdot \vecy_s}$. The scaling matrix is defined as $D(\bDelta) = \text{diag}(e^{2\pi \mathrm{i} \vecl_1 \cdot \bDelta}, \dots, e^{2\pi \mathrm{i} \vecl_v \cdot \bDelta})$. The Fourier coefficients are then recovered via the least-squares solution:
\begin{equation} \label{eq:unified_solution}
    \begin{pmatrix}
        \mathcal{B}_N(f, \vecl_1, \{\vecy_s\},\bDelta) \\
        \vdots \\
        \mathcal{B}_N(f, \vecl_v, \{\vecy_s\},\bDelta)
    \end{pmatrix}
    := D(-\bDelta) (B^H B)^{-1} B^H \mathbf{b}(\bDelta).
\end{equation}
Observe that in the deterministic case ($\bDelta = \vecz$), $D(\vecz) = I$, and the estimator reduces to the standard least-squares form.

Finally, the global approximation of $f$ is given by
\begin{equation}\label{eq:appf}
    f \approx \mathcal{A}_{\bDelta}(f) := \sum_{j=1}^{J} \sum_{\veck \in \Gamma_{\alpha,\vecgamma,N}(\vecl^{(j)}, \vecg)} \mathcal{B}_N(f, \veck,\{\vecy_s\},\bDelta) e^{2\pi\mathrm{i} \veck \cdot \vecx},
\end{equation}
where the set of fibers $\{\Gamma_{\alpha,\vecgamma,N}(\vecl^{(j)}, \vecg)\}_{j=1}^J$ forms a partition of $\setA$. This approach reduces computational overhead by requiring only one lattice rule evaluation per shift for the entire fiber.

\begin{remark}[Algorithmic Advantage via FFT]
    In our definition of the observables, we evaluate the kernel $e^{-2\pi \mathrm{i} \vecl \cdot (n\vecg/N)}$ rather than the fully shifted kernel $e^{-2\pi \mathrm{i} \vecl \cdot (n\vecg/N + \vecy_s + \boldsymbol{\Delta})}$. While mathematically equivalent to the formulation in \cite[Section 3.2]{cai2025}, this structural choice yields a practical advantage: the computation of the observation vector $\mathbf{b}(\boldsymbol{\Delta})$ collapses into a standard Discrete Fourier Transform (DFT) of the shifted spatial samples. This allows the integration of the Fast Fourier Transform (FFT) to compute the observables for all frequencies concurrently in $\mathcal{O}(N \log N)$ operations, strictly avoiding any additional frequency-by-frequency post-processing.
\end{remark}

\subsection{Existence of Good Shifts: A Probabilistic Baseline}
Before delving into our primary deterministic construction, we briefly establish the theoretical performance of the proposed uniform shift strategy using a probabilistic argument. This probabilistic baseline demonstrates the immediate potential of the uniform shift strategy to fundamentally reduce sampling costs compared to state-of-the-art randomized methods.

To ensure a unique solution of \eqref{eq:unified_solution}, we require that the Gram matrix $G_j = B_j^H B_j$ for each fiber is invertible. As established in Remark \ref{rem:trivial_d1}, the trivial case where the maximum fiber length is $R=1$ implies an absolute absence of frequency aliasing. In such a scenario, the Gram matrix collapses to the scalar $S$, and the system is trivially non-singular by assigning a single unshifted lattice ($S=1$ and $\vecy_1 = \vecz$). Given our global assumption $d \ge 2$, we henceforth focus exclusively on the non-trivial case where $R \ge 2$.


One can follow the proof technique in \cite[Lemma 3.5]{cai2025}, but the pessimistic summation of absolute values in the Gershgorin Circle Theorem would yield the bound $S = \mathcal{O}(R^2 \log N)$, which fails to capture the advantages of the uniform shift strategy. Here, we propose a fundamentally different structural approach. By invoking matrix concentration inequalities, we rigorously establish that $S = \mathcal{O}(R \log N)$.

\begin{lemma}\label{lem:prob}
    For any $t \in (0,1)$ and a constant $K > 1$, let $S = \lceil 2KR\log(N)/t^2 \rceil$. Let $N$ be a prime number and $\vecg$ be constructed by Algorithm \ref{alg:CBCg}. If the shifts $\{\vecy_s\}_{s=1}^S \subset [0,1)^d$ are chosen i.i.d. uniformly at random, then
    $$
    \mathbb{P} \left( \forall j = 1, \dots, J, \; \lambda_{\min}(G_j) \ge S(1-t)\right) \ge 1 - N^{1-K},
    $$
    where $G_j$ is the corresponding Gram matrix for the fiber $\Gamma_{\alpha,\gamma,N}(\vecl^{(j)}, \vecg)$.
\end{lemma}

\begin{proof}
    For a fixed $j = 1,2,\ldots,J$, let $G_j = \sum_{s=1}^S X_s$ where $X_s = \vecb_s \vecb_s^H$ and $\vecb_s = [e^{-2\pi \mathrm{i} \vecl_1 \cdot \vecy_s}, \dots, e^{-2\pi \mathrm{i} \vecl_v \cdot \vecy_s}]^\top.$ Since $\vecy_s$ is uniformly distributed, $\mathbb{E}[X_s] = I_{v_j}$. Furthermore, $\lambda_{\max}(X_s) = \|\vecb_s\|_2^2 = {v_j}$ and $\lambda_{\min}(X_s) = 0$. Defining the zero-mean matrix $Z_s = I_{v_j} - X_s$, we have $\lambda_{\max}(Z_s) \le 1$ and $\mathbb{E}[Z_s^2] = ({v_j}-1)I_{v_j}$. Thus, the variance parameter is bounded by $\sigma^2 = \| \sum_{s=1}^S \mathbb{E}[Z_s^2] \|_2 = S({v_j}-1) \le S(R-1)$.

    Applying the Matrix Bernstein inequality \cite[Theorem 1.4]{tropp2012user}, for any $\epsilon \in (0,1)$, we have
    $$
    \begin{aligned}
        \mathbb{P} ( \lambda_{\min}(G_j-SI_{v_j}) \le -\epsilon ) = \mathbb{P} \seq{ \lambda_{\max}\seq{\sum_{s=1}^S Z_s} \ge \epsilon}\le {v_j} \cdot \exp \left( \frac{-\epsilon^2/2}{S(R-1) + \epsilon/3} \right)
    \end{aligned}
    $$
Applying a union bound over all $J$ fibers, and let $\epsilon = tS$ for any $t\in(0,1)$, we have
$$
\mathbb{P} ( \exists j :\lambda_{\min}(G_j-SI) \le -tS ) \le \sum_{j=1}^Jv_j \cdot \exp \left( \frac{-S}{2R/t^2} \right)\le N\exp \left( \frac{-S}{2R/t^2} \right),
$$
where we use the fact that $\sum_{j=1}^Jv_j = |\setA| \le N$.
Thus, 
$$
\mathbb{P} ( \forall j:\lambda_{\min}(G_j) > S(1-t) ) \ge \mathbb{P} ( \forall j :\lambda_{\min}(G_j-SI) > -tS ) \ge 1-N\exp \left( \frac{-S}{2R/t^2} \right),
$$
where we used $\lambda_{\min}(G_j-SI) = \lambda_{\min}(G_j) + \lambda_{\min}(-SI) = \lambda_{\min}(G_j) - S$.
This ensures that all $G_j$ are non-singular with the specified probability. The result follows.
\end{proof}

\begin{remark}[Algorithmic Improvement via Uniform Shifts]\label{rem:prob_comparison}
This probabilistic baseline explicitly highlights the structural efficiency gained by transitioning to uniform shifts:
    \vspace{0.5em} \noindent \textbf{Previous Method \cite{cai2025}:} Assigning frequency-dependent shifts requires $S_{\text{old}} = \lceil 2 K R \log N \rceil$ shifts \textit{per frequency}. With a success probability of $1 - N^{1-K}R^2$, the total sampling cost inherently scales as $\mathcal{O}(R^2 \log N)$.
    \vspace{0.5em} \noindent \textbf{Proposed Uniform Strategy:} By sharing the shift across the fiber, the required total number of shifts is $S_{\text{new}} = \lceil 2 K R \log N / t^2 \rceil$ with an improved success probability of $1 - N^{1-K}$. 
Our uniform shift formulation achieves a rigorous reduction in sampling complexity from quadratic $\mathcal{O}(R^2 \log N)$ to linear $\mathcal{O}(R \log N)$ with respect to the maximum fiber length $R$, while simultaneously removing the $R^2$ degradation from the probability bound.
\end{remark}

\section{The deterministic construction of multiple shifts}
We now provide an explicit construction of multiple shifts $\{\vecy_s\}_{s=1}^S$ that guarantees the invertibility of the Gram matrix $G_j$ for each $j = 1, \dots, J$. The entries of $G = G_j \in \mathbb{C}^{v_j \times v_j}$ are given by $G_{i,i} = S$ and, for $i \neq i'$,
$$
G_{i,i'} = \sum_{s=1}^S \exp(2\pi \mathrm{i} (\vecl_{i'} - \vecl_i) \cdot \vecy_s).
$$
Let $\vech = \vecl_{i'} - \vecl_i \neq \vecz$. By Gershgorin's circle theorem, every eigenvalue $\lambda$ of $G$ lies in the disk:
\begin{equation}\label{eq:lambda}
    |\lambda - S| \le \sum_{i' \neq i} |G_{i,i'}| \le (v-1) \max_{i \neq i'} \left| \sum_{s=1}^S \exp(2\pi \mathrm{i} \vech \cdot \vecy_s) \right| \le (R-1) \max_{\vech \in \setH\setminus \{\vecz\}} \left| \sum_{s=1}^S \exp(2\pi \mathrm{i} \vech \cdot \vecy_s) \right| .
\end{equation}
where $\setH$ is the set of all possible aliasing frequency differences within any fiber:
$$
\setH := \bigcup_{j=1}^J \{\vecl_{i'} - \vecl_i : \vecl_{i},\vecl_{i'} \in \Gamma_{\alpha,\vecgamma,N}(\vecl^{(j)}, \vecg)\}.
$$
\begin{lemma}\label{lem:H}
    If $\vecg$ is constructed by the CBC Algorithm \ref{alg:CBCg}, then any $\vech \in \setH\setminus \{\vecz\}$ satisfies
    \begin{equation}\label{eq:H}
        \max_{1\le j\le d} |h_j| < (2C_\alpha M)^{1/\alpha},
    \end{equation} 
    where $C_\alpha = 1$ for $1/2 < \alpha \le 1$, and $C_\alpha = 2^{\alpha-1}$ for $\alpha \ge 1$.
\end{lemma}
\begin{proof}
    For any $a, b \in \mathbb{R}$, the inequality $|a - b|^\alpha \le C_\alpha(|a|^\alpha + |b|^\alpha)$ holds. Given $\vech = \vecl_{i'} - \vecl_i$ for $\vecl_i, \vecl_{i'} \in \setA$, direct calculation combined with $r_{\alpha, \vecgamma}(\vecl) < M$ and $\gamma_j \le 1$ 
    yields 
    $$
    |h_j|^\alpha \le \frac{|h_j|^\alpha}{\gamma_j} \le C_\alpha \left( \frac{|\ell_{i,j}|^\alpha}{\gamma_j} + \frac{|\ell_{i',j}|^\alpha}{\gamma_j} \right) \le 2C_\alpha M.
    $$
    Taking the $1/\alpha$-th power on both sides completes the proof.
\end{proof}
To ensure the non-singularity of $G$, it suffices to construct $\{\vecy_s\}_{s=1}^S$ such that the exponential sum in \eqref{eq:lambda} is strictly bounded.

\begin{lemma}\label{lem:G}
    Suppose there exists $\{\vecy_s\}_{s=1}^{S}$ such that for any $\vech \in \setH \setminus \{\vecz\}$ and a constant $t \in (0,1)$,
    \begin{equation}\label{eq:ysum}
        \left|\frac{1}{S}\sum_{s=1}^S\exp(2\pi \mathrm{i} \vech\cdot \vecy_s)\right| \le \frac{t}{R-1}.
    \end{equation}
    Then for each $j = 1, \dots, J$, the Gram matrix $G_j$ satisfies:
\begin{itemize}
    \item $0 < (1-t)S \le \lambda(G_j) \le (1+t)S$.
\item $\left\| G_j^{-1} \right\|_2 \le \frac{1}{(1-t)S}$.
\item The condition number $\kappa(G_j) \le \frac{1+t}{1-t}$.
\end{itemize} 
\end{lemma}
\begin{proof}
    The bounds follow directly from substituting assumption \eqref{eq:ysum} into the Gershgorin disk \eqref{eq:lambda}, which trivially yields the bounds for the eigenvalues, inverse norm, and condition number.
\end{proof}

\subsection{Deterministic Construction in Low Dimensions}\label{sec_weil_shifts}
To bound the exponential sum in \eqref{eq:lambda}, we employ a shift structure based on algebraic curves over finite fields, relying on the classical Weil bound for character sums.

\begin{lemma}\label{lem:weilbd}
    Let $p$ be a prime number and $d \in \mathbb{N}^{+}$. If $p \ge d$ and $\vech \not\equiv \vecz \pmod {p}$, we have
    \begin{equation*}
        \left| \sum_{x=0}^{p-1} \exp\left( \frac{2\pi \mathrm{i}}{p} \sum_{i=1}^d h_i x^i \right) \right| \le (d-1)\sqrt{p}.
    \end{equation*}
\end{lemma}
Based on this bound, we construct the shift set using a polynomial curve, which is particularly efficient in low-dimensional cases. We refer to this construction as the \textit{polynomial} strategy.

\begin{lemma}\label{lem:ysweil}
    Let $\vecg$ be constructed by Algorithm \ref{alg:CBCg}, and let $R \ge 2$ denote the maximum fiber length. There exists a prime $p = \mathcal{O}(\max\{d^2 R^2, M^{1/\alpha}\})$ such that for the shift set of size $S=p$ defined as
    \begin{equation}\label{eq:ysweil}
        \{\vecy_s\}_{s=0}^{p-1} = \left\{ \frac{1}{p}\left( s, s^2, \dots, s^d \right) \pmod 1 \;\middle|\; s = 0, \dots, p-1 \right\},
    \end{equation}
    the exponential sum satisfies
    $$
        \left|\frac{1}{p}\sum_{s=0}^{p-1}\exp(2\pi \mathrm{i} \vech\cdot \vecy_s)\right| \le \frac{t}{R-1},
    $$
    for any $\vech \in \setH \setminus \{\vecz\}$ and $t \in (0,1)$.
\end{lemma}

\begin{proof}
    By the definition of $\{\vecy_s\}$, we have
    \begin{equation}\label{eq:exp_sum_decomp}
        \left|\frac{1}{p} \sum_{s=0}^{p-1} \exp(2\pi \mathrm{i} \vech \cdot \vecy_s)\right| = \left|\frac{1}{p} \sum_{s=0}^{p-1} \exp\left( \frac{2\pi \mathrm{i}}{p} \sum_{i=1}^d h_i s^i \right) \right|.
    \end{equation}
    For any difference vector $\vech \in \setH \setminus \{\vecz\}$, let $D \in \{1, \dots, d\}$ be the largest index such that $h_D \neq 0$. By Lemma \ref{lem:H}, $|h_D| < (2C_\alpha M)^{1/\alpha}$. Choosing any prime $p > \max\{(2C_\alpha M)^{1/\alpha}, d\}$ ensures both $h_D \not\equiv 0 \pmod p$ and $p > D$. Applying Lemma \ref{lem:weilbd} yields:
    $$
        \left|\frac{1}{p} \sum_{s=0}^{p-1} \exp\left( \frac{2\pi \mathrm{i}}{p} \sum_{i=1}^d h_i s^i \right) \right| \le \frac{(D-1)\sqrt{p}}{p} \le \frac{d-1}{\sqrt{p}}.
    $$
    To satisfy the non-singularity threshold $\frac{t}{R-1}$, we require $\frac{d-1}{\sqrt{p}} \le \frac{t}{R-1}$, which implies $p \ge \frac{(d-1)^2(R-1)^2}{t^2}$. 
    
    Therefore, any prime $p$ satisfying $p > \max\left\{ (2C_\alpha M)^{1/\alpha}, d, \frac{(d-1)^2(R-1)^2}{t^2} \right\}$ algebraically satisfies the condition. By Bertrand's postulate, a prime always exists within a factor of 2 of this lower bound, yielding $p = \mathcal{O}(\max\{d^2 R^2, M^{1/\alpha}\})$.
\end{proof}

\subsection{Deterministic Construction via the Union of Rank-1 Lattices with a Common Generator}\label{sec_union_shifts}
To overcome the dimensional dependence of polynomial curves, we employ a geometric projection strategy. Instead of utilizing a single large prime, we construct the shift set using a common integer projection vector $\bz \in \mathbb{Z}^d$ scaled by $k$ consecutive small primes $p_1, \dots, p_k$:
\begin{equation}\label{eq:ysmulti}
    \{\vecy_s\}_{s=1}^S = \bigcup_{j=1}^k \mathcal{L}_j(\bz, p_j), \quad \text{where } \mathcal{L}_j(\bz, p_j) = \left\{ \left\{ \frac{i \bz}{p_j} \right\} : i = 0, \dots, p_j-1 \right\}.
\end{equation}
The total number of shifts is $S = \sum_{j=1}^k p_j$. This shared-direction multi-prime lattice enables the use of the Chinese Remainder Theorem (CRT) to satisfy the exponential sum threshold in \eqref{eq:ysum}. We refer to this construction as the \textit{multi-lattice} strategy.

To use the CRT, the vector $\bz$ must satisfy $\vech \cdot \bz \neq 0$ for all $\vech \in \setH \setminus \{\vecz\}$. Classical lattice rules enforce a stronger modular condition $\vech \cdot \bz \not\equiv 0 \pmod p$ \cite{kammerer2013}. Relaxing this to a simple non-zero integer dot product avoids modular constraints and simplifies the construction.

Let $H \in \mathbb{Z}^{(|\setH|-1) \times d}$ be the matrix whose rows are the vectors $\vech_i \in \setH \setminus \{\vecz\}$. For each row $\vech_i$, we define its active column index $I(i) = \max \{ j : h_{i,j} \neq 0 \}$ as the position of its rightmost non-zero entry. We partition the rows into $d$ groups $H_1, \dots, H_d$, where $H_j$ contains rows ending exactly at column $j$. 

For any $\vech_i \in H_j$, the condition $\vech_i \cdot \bz \neq 0$ depends only on the first $j$ components of $\bz$:
\begin{equation*}
    \vech_i \cdot \bz = \sum_{l=1}^{j-1} h_{i,l} z_l + h_{i,j} z_j + \sum_{l=j+1}^{d} h_{i,l} z_l\neq 0 \implies  z_j \neq -\frac{1}{h_{i,j}} \sum_{l=1}^{j-1} h_{i,l} z_l.
\end{equation*}
Since $z_j$ is an integer, this inequality only imposes a restriction if the right-hand side is an integer (i.e., if $h_{i,j}$ divides the sum). Because the constraints in $H_j$ are decoupled from $z_{j+1}, \dots, z_d$, we can construct $\bz$ sequentially. At step $j$, with $z_1, \dots, z_{j-1}$ fixed, $z_j$ merely needs to avoid at most $|H_j|$ forbidden integer values. Algorithm \ref{alg:CBC_z} implements this by searching a symmetric sequence of integers.
\vspace{-3mm}
\begin{algorithm}[H]
\caption{Symmetric Component-by-Component Construction}
\label{alg:CBC_z}
\begin{algorithmic}[1]
\Require Matrix $H \in \mathbb{Z}^{(|\setH|-1) \times d}$ with no zero rows.
\Ensure Vector $\bz \in \mathbb{Z}^d$ such that $(H\bz)_i \neq 0$ for all $i$.
\State Partition $H$ into active sets $H_j \leftarrow \{ i : I(i) = j \}$ for $j=1, \dots, d$.
\State Initialize $\bz \leftarrow \vecz$.
\For{$j = 1$ to $d$}
    \State $\mathcal{F}_j \leftarrow \emptyset$ \Comment{Set of forbidden values}
    \For{each $i \in H_j$}
        \State $C_i \leftarrow \sum_{l=1}^{j-1} h_{i,l} z_l$ 
        \If{$C_i \equiv 0 \pmod{h_{i,j}}$}
            \State $\mathcal{F}_j \leftarrow \mathcal{F}_j \cup \{ -C_i / h_{i,j} \}$
        \EndIf
    \EndFor
    \State $\mathcal{C} \leftarrow \{0, 1, -1, \dots, \lceil |H_j|/2 \rceil, -\lceil |H_j|/2 \rceil\}$
    \State $z_j \leftarrow \text{first } x \in \mathcal{C} \text{ such that } x \notin \mathcal{F}_j$
\EndFor
\State \Return $\bz$
\end{algorithmic}
\end{algorithm}
\vspace{-3mm}
We now establish the sample size bound $S$ under this CRT-based construction.

\begin{theorem}\label{thm:ys}
    Let $\bz$ be generated by Algorithm \ref{alg:CBC_z}, and let $p_1 < p_2 < \cdots < p_k$ be $k$ consecutive primes. Define the one dimensional projection bound $V = \max\limits_{\vech \in \setH} |\vech \cdot \bz|$. For any $t \in (0,1)$, if
    \begin{equation*}
        p_1 \ge \frac{2(R-1) \ln(V)}{c t} \quad \text{and} \quad k = \left\lceil \frac{2(R-1) \ln(V)}{t \ln p_1} \right\rceil,
    \end{equation*}
    where $c = 0.32$, then the exponential sum threshold \eqref{eq:ysum} holds. The total sample size scales as:
    \begin{equation*}
        S = \sum_{j=1}^k p_j \le \frac{8(R-1)^2 (\ln V)^2}{ct^2 \ln (4(R-1) \ln V)} = \mathcal{O}\left(\frac{R^2(\ln(dRN^2))^2}{\ln(R\ln(dRN^2))}\right).
    \end{equation*}
\end{theorem}

\begin{proof}
    For any $\vech \in \setH \setminus \{\vecz\}$, let $X_{\vech} = |\vech \cdot \bz| > 0$ be the collision integer. The exponential sum over $\mathcal{L}_j(\bz, p_j)$ equals $p_j$ if $p_j$ divides $X_{\vech}$ (a ``bad'' prime), and exactly $0$ otherwise due to the orthogonality of characters \cite{dick2022}.

    Let $m$ denote the number of ``bad'' primes for a given $\vech$. If $p_j \in \mathcal{P}_{bad}$, it implies $p_j$ divides the scalar $X_{\vech}$. Since $X_{\vech} \neq 0$ and the primes are mutually coprime, the Chinese Remainder Theorem (CRT) \cite{hardy1979} dictates that the product of these $m$ distinct ``bad'' primes must also divide $X_{\vech}$. Consequently,
    \begin{equation*}
        p_1^m \le \prod_{p_j \in \mathcal{P}_{bad}} p_j \le |X_{\vech}| \le V \implies m \le \ln(V)/\ln (p_1).
    \end{equation*}
     Defining the prime ratio $\rho := p_k/p_1$, the exponential sum is bounded by:
    \begin{equation*}
        \left| \frac{1}{S} \sum_{s=1}^S \exp(2\pi \mathrm{i} \vech \cdot \vecy_s) \right| = \frac{1}{S} \sum_{p_j \mid X_{\vech}} p_j \le \frac{m p_k}{k p_1} = \frac{m}{k}\rho \le \frac{\rho \ln V}{k \ln p_1}.
    \end{equation*}
    To ensure this is bounded by $\frac{t}{R-1}$, we require $k \ge \frac{(R-1)\rho \ln V}{t \ln p_1}$. By the prime counting bounds from Rosser and Schoenfeld \cite{rosser1962}, for any $n \ge 2$, $\pi(2n) - \pi(n) > c \frac{n}{\ln n}$ with $c=0.32$, where $\pi(n)$ is the number of primes less than or equal to $n$. The number of primes $p$ in the interval $[p_1, 2p_1)$ is greater than $c \frac{p_1}{\ln p_1}$. By choosing $p_1$ such that $c \frac{p_1}{\ln p_1} \ge \frac{2(R-1)  \ln V}{t \ln p_1}$, i.e. $p_1 \ge \frac{2 (R-1) \ln V}{c t}$, there are at least $\frac{2 (R-1) \ln V}{t \ln p_1}$ primes in the interval $[p_1, 2p_1)$. Setting $k = \left\lceil \frac{2(R-1) \ln V}{t \ln p_1} \right\rceil$ thereby ensures that the $k$-th prime satisfies $p_k \le 2 p_1$, effectively locking the ratio $\rho \le 2$. This strictly satisfies the required condition $k \ge \frac{(R-1) \rho \ln V}{t \ln p_1}$.

    The upper bound on $S$ follows from $S \le k p_k < 2 k p_1$. Since Algorithm \ref{alg:CBC_z} implies 
    $$
        \max\limits_{1\le j\le d} |z_j| \le \max\limits_{1\le j\le d}{\lceil |H_j|/2 \rceil} \le  (|\setH|-1)/2+1 \le (R-1)N/4+1 < RN/4,
    $$
    and $\max\limits_{1\le j\le d} |h_j| < (2C_\alpha M)^{1/\alpha}$ by \eqref{eq:H}, we have 
    $$
        V = \max_{\vech \in \setH}|\vech \cdot \bz| \le  d\max\limits_{1\le j\le d} |h_j| \max\limits_{1\le j\le d} |z_j| \le d(2C_\alpha M)^{1/\alpha}RN/4.
    $$
    Substituting $V$ into the bound yields the asymptotic scaling.
\end{proof}

\begin{remark}
The multi-prime structure of the shift set $\bigcup_{j=1}^k \mathcal{L}_j(\bz, p_j)$ is geometrically analogous to the multiple rank-1 lattices in \cite{gross2021}. However, the two frameworks differ fundamentally in their construction logic and algebraic objectives. While \cite{gross2021} relies on an adaptive prime selection and a generator $\bz$ derived from a pre-existing large lattice to ensure absolute collision avoidance (i.e., $\vech \cdot \bz \not\equiv 0 \pmod{p_j}$ for at least one $p_j$), our framework utilizes a self-contained $\bz$ satisfying only $\vech \cdot \bz \neq 0$ and employs consecutive small primes. In Theorem \ref{thm:ys}, the CRT is used solely to bound the number of ``bad primes" dividing the projection scalar $\vech \cdot \bz$, which ensures the diagonal dominance of the Gram matrix and structurally compresses the shift count $S$ to a logarithmic scale.
\end{remark}

\subsection{Adaptive Deterministic Strategy}
While Theorem \ref{thm:ys} provides a rigorous upper bound $S_{\mathrm{CRT}} = \mathcal{O}(R^2 (\ln V)^2)$ where $V = \mathcal{O}(d R M^{1/\alpha} N)$, this bound for $V$ is often conservative. It assumes a worst-case scenario where the collision integer $X_{\vech} = |\vech \cdot \bz|$ simultaneously reaches the theoretical upper limits of both $\vech$ and $\bz$, i.e. we use the estimation $|\setH|-1 \le (R-1) N/2$. In practice, the absolute projection bound $V = \max_{\vech \in \setH} |X_{\vech}|$ is strictly bounded well below this limit.

Crucially, the vector $\bz$ generated by Algorithm~\ref{alg:CBC_z} is generated prior to the selection of the prime $p$. This two-step decoupled approach circumvents the dimensional bottleneck observed in standard component-by-component constructions \cite{kammerer2013}. Standard methods attempt to determine the generating vector $\bz$ and the prime $p$ simultaneously, requiring an evaluation over the full $d$-dimensional difference set $\setH$, which leads to a combinatorial explosion in high dimensions. In contrast, our algorithm orthogonally decouples this procedure by first fixing a universal projection direction $\bz$. This maps the $d$-dimensional aliasing constraints into a one-dimensional scalar set $\{X_{\vech}\}$, isolating the prime search to the scalar regime $p \le V$ and eliminating the high-dimensional search space. Consequently, a greedy search for a single prime $p$ satisfying $X_{\vech} \not\equiv 0 \pmod p$ frequently succeeds before exhausting the $S_{\mathrm{CRT}}$ budget. This corresponds to the special case of $k = 1$ in the multi-lattice strategy \eqref{eq:ysmulti}, which we refer to as the single-lattice strategy.

Similarly, while Lemma \ref{lem:ysweil} establishes a prime bound $\mathcal{O}(\max\{d^2 R^2, M^{1/\alpha}\})$ via the Weil bound for low-dimensional constructions, this theoretical estimate serves as a worst-case envelope. Practical implementations typically identify valid primes satisfying the exponential sum threshold well below this theoretical bound.

To formalize the mathematical boundaries of the single-lattice strategy before analyzing its dimensional behavior, we establish its underlying geometric capacity constraint.

\begin{lemma}[Geometric Capacity of the Single-Lattice Strategy] \label{lem:capacity}
    Suppose the shift set is constructed using a single prime $p$ and a projection vector $\bz$ such that $\vech \cdot \bz \not\equiv 0 \pmod p$ for all $\vech \in \mathcal{H}\setminus\{\vecz\}$. This prime is strictly subject to the geometric capacity lower bound:
    \begin{equation}\label{eq:lowerp}
        p \ge \frac{\lfloor(\gamma_1 M)^{1/\alpha}\rfloor \lfloor(\gamma_2 M)^{1/\alpha}\rfloor}{N}.
    \end{equation}
\end{lemma}

\begin{proof}
    It is clear that the combined sampling set $\{n\vecg/N + s\bz/p\}_{n=0,\dots,N-1}^{s=0,\dots,p-1}$ structurally collapses into a single rank-1 lattice of size $N_{tot} = Np$. Let $\vecg_{tot}$ be the equivalent generator of the combined rank-1 lattice. By the Chinese Remainder Theorem, $\vecg_{tot} \equiv \vecg \pmod{N}$ and $\vecg_{tot} \equiv \bz \pmod{p}$. For any aliasing difference vector $\vech \in \mathcal{H}\setminus\{\vecz\}$, the base lattice property yields $\vech \cdot \vecg \equiv 0 \pmod{N}$.
    
    Assume for contradiction that $\vech \cdot \vecg_{tot} \equiv 0 \pmod{Np}$. This strictly implies $\vech \cdot \vecg_{tot} \equiv 0 \pmod{p}$. Since $\vecg_{tot} \equiv \bz \pmod{p}$, it follows that $\vech \cdot \bz \equiv 0 \pmod{p}$, which directly contradicts our search condition. Thus, $\vech \cdot \vecg_{tot} \not\equiv 0 \pmod{Np}$ holds for all $\vech \in \mathcal{H}\setminus\{\vecz\}$, ensuring no difference vectors belong to the dual lattice of the combined set. According to the geometric capacity bounds of rank-1 lattices \cite[Lemma 2.1]{kammerer2013}, preventing such dual lattice intersections necessitates $Np \ge \lfloor(\gamma_1 M)^{1/\alpha}\rfloor \lfloor(\gamma_2 M)^{1/\alpha}\rfloor$, which yields \eqref{eq:lowerp}.
\end{proof}

We now analyze the algebraic mechanisms of these complementary construction strategies across varying regimes, which directly motivates their integration into the proposed unified Adaptive Framework.

    \vspace{0.5em} \noindent \textbf{Low-Dimensional Setting:} Ensuring the cardinality of the hyperbolic cross set $|\setA| \approx N$ requires a large truncation threshold $M$. This permits the aliasing difference vectors $\vech \in \setH\setminus\{\vecz\}$ to possess large individual components, which directly leads to a large projection bound $V$. The resulting integers in $\{X_{\vech}\}$ accumulate a wide union of distinct prime factors. Consequently, identifying a single valid prime $p$ for the single-lattice strategy requires an excessively large search space.
    
    In contrast, the polynomial strategy \eqref{eq:ysweil} structurally alters the failure condition. To violate the non-degenerate condition $\vech \not\equiv \vecz \pmod p$, a prime must simultaneously divide every component of $\vech$. Algebraically, this implies $p \mid \gcd(h_1, \dots, h_d)$. Because $\gcd(h_1, \dots, h_d) \le \min_{h_j \neq 0} |h_j|$, this intersection property functions as a strict filter for potential ``bad'' primes, rendering the polynomial search insensitive to the extreme magnitude of $\max |h_j|$. Furthermore, in low dimensions, the Weil bound threshold ensuring the exponential sum decay remains small (e.g., $(d-1)^2(R-1)^2 = 4$ for $d=2, R=3$). The combination of this easily satisfied threshold and the strict $\gcd$ filter allows the polynomial strategy to identify valid shifts using small primes, bypassing the theoretical bound $p > \max|h_j|$.
    
    \vspace{0.5em} \noindent \textbf{High-Dimensional Setting:} To maintain the computational budget in high dimensions, the threshold $M$ shrinks severely, naturally bounding the components of $\vech$ and resolving the prime accumulation issue. By Lemma \ref{lem:capacity}, this severe shrinkage aggressively drives down the numerator in \eqref{eq:lowerp}, formally stabilizing the single-lattice strategy at remarkably small primes. Algorithmically, the inequality \eqref{eq:lowerp} serves as a highly efficient $\mathcal{O}(1)$ pruning filter: any prime strictly smaller than this threshold mathematically guarantees a collision in $\mathcal{H}$, allowing the algorithm to safely bypass $\mathcal{O}(|\mathcal{H}|)$ divisibility checks. 
    
    Conversely, the polynomial strategy is structurally penalized here: the minimum prime required to ensure the exponential sum satisfies the deterministic threshold is algebraically bounded by the dimension and maximum fiber length (i.e., $p \ge \mathcal{O}(d^2R^2)$), forcing an inevitable inflation of the prime search.

    \vspace{0.5em} \noindent \textbf{Theoretical Limits and Asymptotic Fallback:} While the single-lattice strategy is computationally efficient, it lacks practical relevance in the asymptotic sense. As established in Lemma \ref{lem:capacity}, the sampling set geometrically collapses into a single rank-1 lattice of size $N_{tot} = Np$. According to the Pigeonhole Principle \cite{byrenheid2017tight, dick2022}, this topology is fundamentally restricted by a worst-case lower bound of $\mathcal{O}(N_{tot}^{-\alpha/2})$. Furthermore, to maintain the optimal truncation $M = \mathcal{O}(N^{\alpha})$ as $N \to \infty$, the capacity constraint \eqref{eq:lowerp} inevitably forces $p \ge \mathcal{O}(N)$. Unrestricted, the single-lattice strategy would inflate the total cost to $N_{tot} \ge \mathcal{O}(N^2)$, thereby degrading the optimal spatial error $\mathcal{O}(N^{-\alpha})$ to the suboptimal bound $\mathcal{O}(N_{tot}^{-\alpha/2})$. To decisively resolve this asymptotic barrier, we impose the theoretical multi-prime limit (Theorem \ref{thm:ys}) as a hard search ceiling. As $p$ attempts to grow linearly, the search naturally hits this global bound, triggering an automatic fallback to the multi-prime strategy, which mathematically preserves the optimal worst-case complexity of $\mathcal{O}(N_{tot}^{-\alpha})$ up to logarithmic factors.

    \vspace{0.5em} \noindent \textbf{Greedy Multi-Prime Search:} To strictly accommodate the deterministic worst-case setting across all regimes without immediately defaulting to the excessively large theoretical budget $S_{\mathrm{CRT\_bound}}$ established in Theorem \ref{thm:ys}, we introduce a \textit{greedy multi-prime search} mechanism. Since the multi-lattice strategy reduces the exponential sum criteria to pure integer divisibility checks for $X_{\vech}$, we sequentially accumulate small primes $p_1, \ldots, p_k$ and dynamically evaluate the non-orthogonality ratio. This greedy approach strictly preserves the non-collapsing multi-lattice topology—avoiding the $\mathcal{O}(N_{tot}^{-\alpha/2})$ worst-case rate—while compressing the total sampling cost $S = \sum_j p_j$ to be orders of magnitude smaller than the theoretical upper bound.

Motivated by these complementary algebraic structures and theoretical constraints, we propose a comprehensive Adaptive Shift Construction Framework. The algorithm evaluates these strategies concurrently, tracks a dynamic global upper bound, and dynamically selects the optimal shift construction.

\begin{algorithm}[htbp]
\caption{Adaptive Deterministic Shift Construction}
\label{alg:adaptive_shift}
\begin{algorithmic}[1]
\Require Difference set $\setH$, maximum fiber length $R$, threshold $t \in (0,1)$, Sample Size $N$, threshold $M$, weight $\vecgamma$
\Ensure Number of shifts $S$, and shift set $Y = \{\vecy_0,\dots,\vecy_{S-1}\}$

\If{$R = 1$ \textbf{or} $\setH = \emptyset$}
    \State \Return $S \gets 1$, $Y \gets \{\vecz\}$ 
\EndIf

\State Construct the CRT generating vector $\bz \in \mathbb{Z}^d$ for $\setH$ (Algorithm \ref{alg:CBC_z})
\State Calculate projection integers $X_{\vech} = |\vech \cdot \bz|$ for all $\vech \in \setH\setminus\{\vecz\}$
\State Calculate the theoretical multi-prime sequence $p_1, \dots, p_k$ (Theorem \ref{thm:ys})

\State \textbf{Initialize Global Best:}
\State $S_{\mathrm{best}} \gets \sum_{j=1}^k p_j$
\State $Y_{\mathrm{best}} \gets \bigcup_{j=1}^k \left\{ \frac{s\bz}{p_j} \bmod 1 \;\middle|\; s=0,\dots,p_j-1 \right\}$
\State $p_{\min} \gets \frac{\lfloor(\gamma_1 M)^{1/\alpha}\rfloor \lfloor(\gamma_2 M)^{1/\alpha}\rfloor}{N}$ \Comment{Theoretical threshold \eqref{eq:lowerp}}

\State \textbf{Initialize Greedy Trackers:}
\State $p \gets \text{next prime}(R)$ \Comment{The next prime $\ge R$}
\State $P_{\mathrm{greedy}} \gets \emptyset, \quad S_{\mathrm{greedy}} \gets 0, \quad N_{\vech} \gets 0 \text{ for all } \vech \in \setH\setminus\{\vecz\}$
\State $\mathrm{greedy\_found} \gets \text{False}$

\While{$p < S_{\mathrm{best}}$}
    \State \textbf{Step A: Polynomial Strategy}
    \State $Y_{\mathrm{poly}} \leftarrow \left\{ \frac{1}{p}(s,s^2,\dots,s^d) \bmod 1 \;\middle|\; s=0,\dots,p-1 \right\}$
    \If{$\max_{\vech\in\setH} \left| \sum_{s=0}^{p-1} e^{2\pi \mathrm{i} \vech\cdot\vecy_s} \right| \le \frac{t\,p}{R-1}$}
        \State $S_{\mathrm{best}} \gets p, \quad Y_{\mathrm{best}} \gets Y_{\mathrm{poly}}$
        \State \textbf{break} \Comment{Global minimum reached}
    \EndIf
    
    \State \textbf{Step B: Single-Lattice Strategy}
    \If{ $p \ge p_{\min}$ and $X_{\vech} \not\equiv 0 \pmod p$ for all $\vech \in \setH\setminus\{\vecz\}$}
        \State $S_{\mathrm{best}} \gets p, \quad Y_{\mathrm{best}} \gets \left\{ \frac{s\bz}{p} \bmod 1 \;\middle|\; s=0,\dots,p-1 \right\}$
        \State \textbf{break} \Comment{Global minimum reached}
    \EndIf
    
    \State \textbf{Step C: Multi-Lattice Strategy (Greedy)}
    \If{\textbf{not} $\mathrm{greedy\_found}$}
        \State $P_{\mathrm{greedy}} \gets P_{\mathrm{greedy}} \cup \{p\}, \quad S_{\mathrm{greedy}} \gets S_{\mathrm{greedy}} + p$
        \State For all $\vech$, \textbf{if} $X_{\vech} \equiv 0 \pmod p$ \textbf{then} $N_{\vech} \gets N_{\vech} + p$
        \If{$\max_{\vech} \left( \frac{N_{\vech}}{S_{\mathrm{greedy}}} \right) \le \frac{t}{R-1}$}
            \State $\mathrm{greedy\_found} \gets \text{True}$ \Comment{Freeze greedy accumulation}
            \If{$S_{\mathrm{greedy}} < S_{\mathrm{best}}$}
                \State $S_{\mathrm{best}} \gets S_{\mathrm{greedy}}$
                \State $Y_{\mathrm{best}} \gets \bigcup_{p_j \in P_{\mathrm{greedy}}} \left\{ \frac{s\bz}{p_j} \bmod 1 \;\middle|\; s=0,\dots,p_j-1 \right\}$
            \EndIf
        \EndIf
    \EndIf
    
    \State $p \gets \text{next prime}(p+1)$
\EndWhile

\State \Return $S \gets S_{\mathrm{best}}$, $Y \gets Y_{\mathrm{best}}$
\end{algorithmic}
\end{algorithm}

To systematically validate the internal mechanisms of this adaptive framework, we evaluate $S$ across dimensions $d \in \{2, \dots, 50\}$ and base lattice sizes $N \in \{2^{10}, \dots, 2^{20}\}$, where each $N$ is assigned its nearest prime neighbor within the specified range. The test environment is configured with a low smoothness $\alpha = 1$, a slow-decaying product weights $\gamma_j = 2^{(1-j)/10}$ and $t = 0.95$. In the numerical experiments presented in Section \ref{sec:numer}, to illustrate the error behavior without severe truncation errors, we employ a bisection method to determine $M$ such that $|\setA| \approx N$, rather than adopting its theoretical lower bound. We follow the identical approach here.

\begin{figure}[htbp]
    \centering
    \includegraphics[width=1\linewidth]{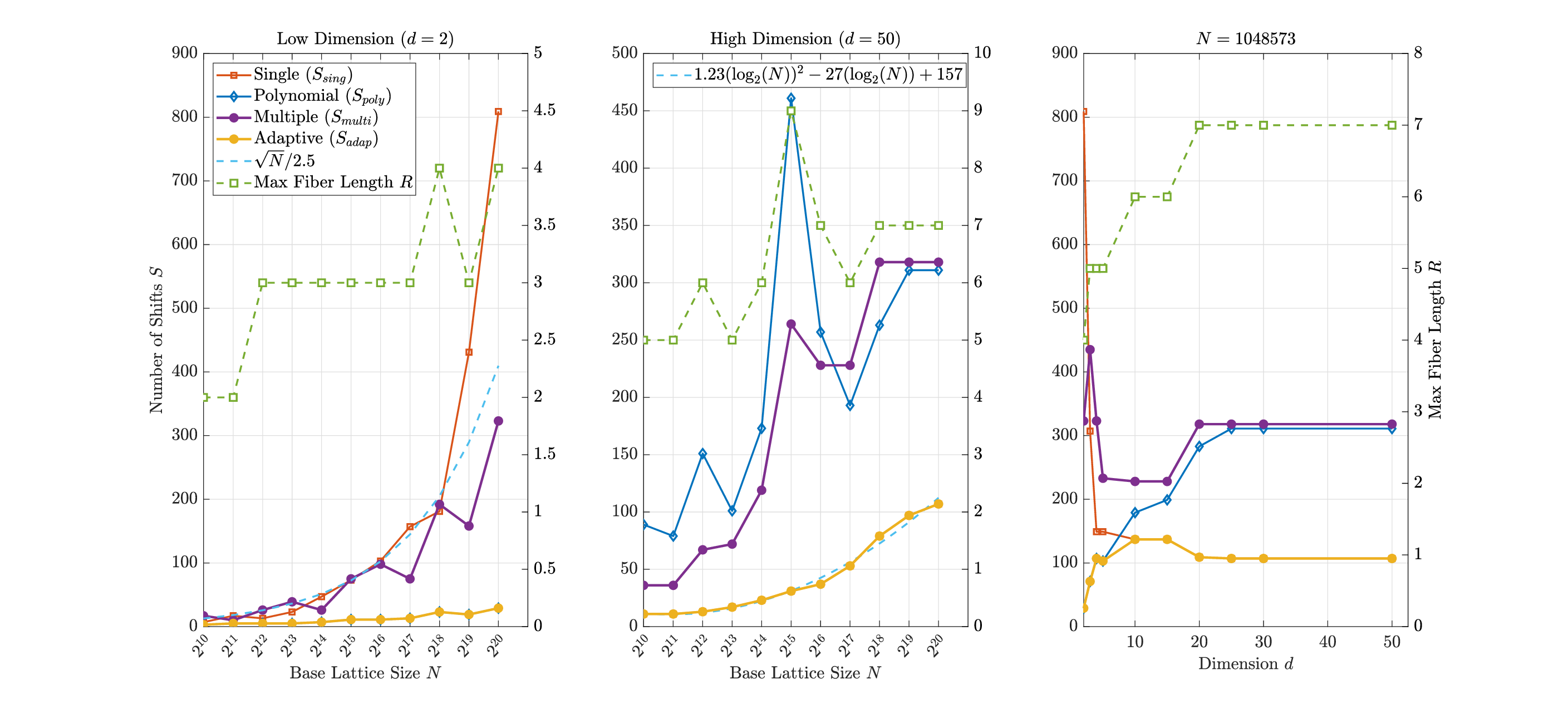}
    \caption{The total number of shifts $S$ evaluated across different strategies.}
    \label{fig:s_optimal}
\end{figure}

Figure \ref{fig:s_optimal} compares the baseline Single Prime Lattice strategy ($S_{\mathrm{sing}}$), the Polynomial strategy ($S_{\mathrm{poly}}$), the greedy Multi-Lattice strategy ($S_{\mathrm{multi}}$), and the proposed unified Adaptive framework ($S_{\mathrm{adap}}$). Note that the theoretical bound $S_{\mathrm{CRT\_bound}}$ is omitted, as its excessively large numerical values would obscure the behavior of the operational strategies. The empirical behavior strictly corroborates our algebraic analysis:

\vspace{0.5em} \noindent \textbf{Low-Dimensional Regime ($d = 2$):} As predicted, the single-lattice strategy ($S_{\mathrm{sing}}$) suffers from severe prime factor accumulation due to the large truncation threshold $M$. This manifests as a substantial $\mathcal{O}(N^{0.5})$ baseline cost within the interval $N \in [2^{10}, 2^{18}]$, followed by a rapid inflation consistent with the theoretical $\mathcal{O}(N)$ limit as $N$ approaches $2^{20}$. By seamlessly tracking the highly efficient polynomial strategy ($S_{\mathrm{poly}} < 30$), the adaptive framework $S_{\mathrm{adap}}$ completely bypasses this massive geometric penalty.

\vspace{0.5em} \noindent \textbf{High-Dimensional Regime ($d = 50$):} Driven by extreme threshold shrinkage, the single-lattice strategy mathematically stabilizes at remarkably small primes, maintaining an empirical $\mathcal{O}((\log_2 N)^2)$ trajectory even at $N \approx 2^{20}$. Conversely, $S_{\mathrm{poly}}$ and $S_{\mathrm{multi}}$ are structurally penalized by the high spatial dimension. Here, $S_{\mathrm{adap}}$ adaptively transitions to the single-lattice construction, locking onto the absolute minimum sampling cost.

\vspace{0.5em} \noindent \textbf{Dimensional Crossover and Global Optimality:} The rightmost panel strictly isolates the dimensional effect by fixing an extreme base size $N \approx 2^{20}$. While $S_{\mathrm{poly}}$ suffers inevitable algebraic inflation as $d$ grows, the required prime for $S_{\mathrm{sing}}$ drops precipitously. The adaptive strategy dynamically captures this precise crossover point without manual tuning, globally minimizing the sampling cost across all parameter regimes.

\begin{remark}[Numerical Stability]\label{rem:stable}
    In \cite{cai2025}, it is proved that assigning $S = \left\lceil 2KR\log N \right\rceil$ random shifts bounds the Gram matrix eigenvalues by
    $
    S \le \lambda( G_j) \le (2v_j-1)S,
    $
    with high probability. Consequently, the probabilistic condition number is bounded by $\kappa(G_j) \le 2R-1 = \mathcal{O}((\log N)^{d-1})$, which theoretically degenerates in high dimensions. In strict contrast, our deterministic approach explicitly enforces $\kappa(G_j) \le \frac{1+t}{1-t}$ globally (Lemma \ref{lem:G}). This structurally guarantees absolute numerical stability independent of $N$ or $d$.
\end{remark}

To validate this robustness, we compare the total shifts $S$ and the maximum condition number $\max_{j}\{\kappa(G_j)\}$ across three approaches: the standard Probabilistic method \cite[Lemma 3.6]{cai2025} (99\% success rate), the Simplified Probabilistic method (Lemma \ref{lem:prob}), and our unified Adaptive strategy, retaining the exact parameter configuration from Figure \ref{fig:s_optimal}.

\begin{figure}[htbp]
    \centering
    \includegraphics[width=1\linewidth]{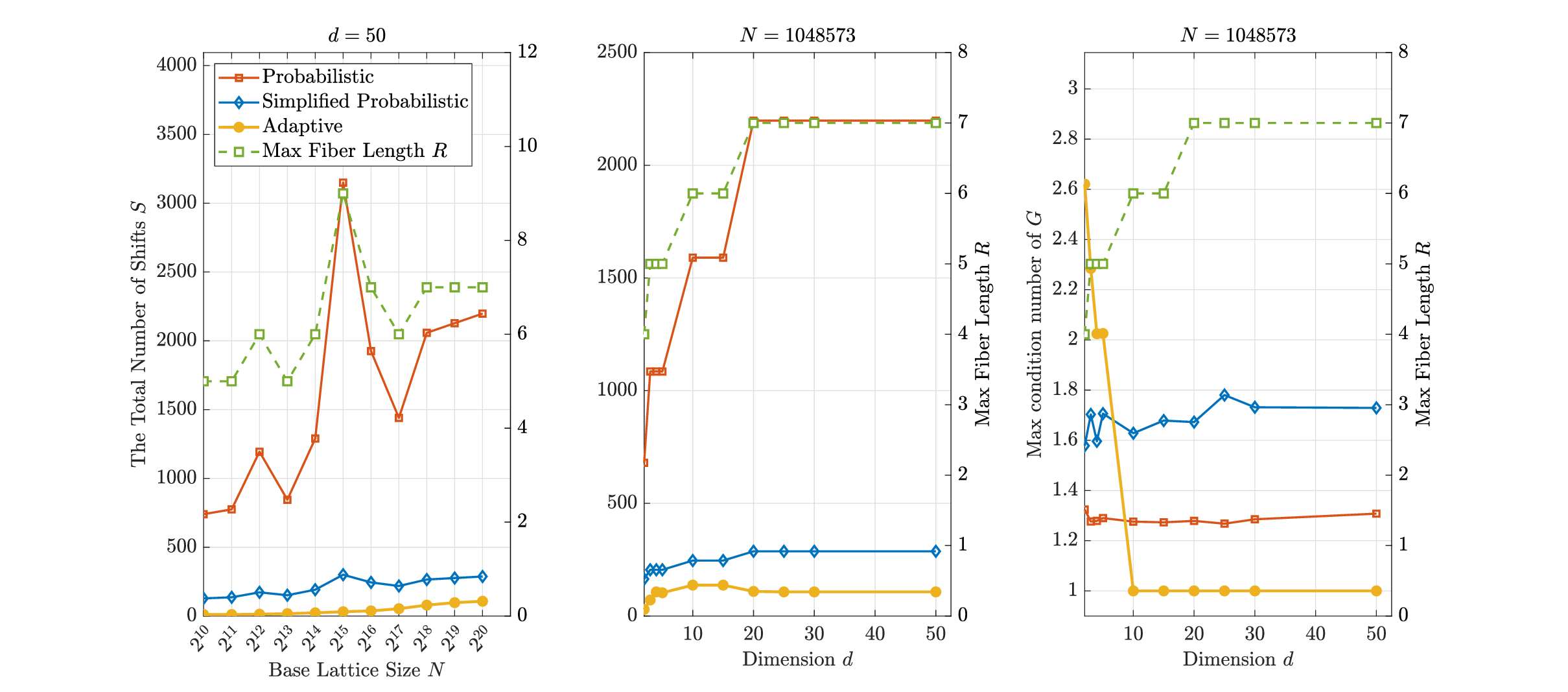}
    \caption{Sampling complexity and numerical stability comparison against state-of-the-art probabilistic methods.}
    \label{fig:compare_S}
\end{figure}
\vspace{0.5em} \noindent \textbf{Order-of-Magnitude Sampling Reduction:} As shown in Figure \ref{fig:compare_S} (left and middle), the standard probabilistic method scales as $\mathcal{O}(R^2)$, severely amplifying the step-like growth of the maximum fiber length $R$. While the Simplified Probabilistic method tempers this to $\mathcal{O}(R \log N)$, the Adaptive strategy dramatically outperforms both. At the extreme scale of $d=50$ and $N \approx 2^{20}$, the probabilistic baseline requires $S = 2198$ shifts to ensure recovery. Our Adaptive strategy guarantees deterministic reconstruction with a mere $S = 107$ shifts, representing a $20\times$ reduction in sampling budget.

\vspace{0.5em} \noindent \textbf{Guaranteed Numerical Stability:} The condition number analysis (Figure \ref{fig:compare_S}, right) explicitly verifies the dynamic transition mechanics. In low dimensions ($d < 10$), $S_{\mathrm{adap}}$ selects the Polynomial curve, yielding a safely bounded condition number ($\kappa \approx 2.6$) strictly within the theoretical limit $\frac{1+t}{1-t} = 39$. As $d \ge 10$, the strategy systematically switches to the Single-Lattice construction. This physically renders the columns of the system matrix perfectly orthogonal ($B^H B = p I_{v_j}$), driving the condition number down to the absolute optimal minimum of $\kappa(G_j) = 1$. The deterministic Adaptive framework thus achieves an order-of-magnitude reduction in sampling cost while mathematically cementing absolute numerical stability.

\section{Theoretical analysis}\label{sec_error_analysis}
We now consider the $L_{\infty}$ error with respect to the total number of function evaluations $N_{tot} = NS$. 
By Theorem \ref{thm:ys}, we have 
$$
S = \mathcal{O}\left(\frac{R^2(\ln (dRN^2))^2}{\ln (R\ln(dRN^2))}\right).
$$
By Lemma \ref{lem:R}, the maximum fiber length satisfies $R \le \mathcal{O}((\log N)^{d-1})$. Hence, 
$$
N \le N_{tot} = NS = \mathcal{O}\left(\frac{NR^2(\ln (dRN^2))^2}{\ln (R\ln(dRN^2))}\right) \le \mathcal{O}\left( N (\log N)^{2d} \right).
$$
Thus, there exists a constant $C > 0$ such that
\begin{equation}\label{eq:Ntot}
    \frac{C N_{tot}}{(\log N_{tot})^{2d}} \le N \le N_{tot}.
\end{equation}

We analyze the worst-case approximation error defined as
$$
e_{\infty}^{\text{wor}}(\mathcal{A}, \mathcal{K}_{d,\alpha,\vecgamma}) = \sup_{\substack{f \in \mathcal{K}_{d,\alpha,\vecgamma} \\ \|f\|_{\mathcal{K}_{d,\alpha,\vecgamma}} \le 1}} \sup_{\vecx \in [0,1]^d} |f(\vecx) - \mathcal{A}(f)(\vecx)|,
$$
where $\mathcal{A}(f) = \mathcal{A}_{\vecz}(f)$ is given by \eqref{eq:appf}.

We first state a lemma from \cite[Lemma 14.2]{dick2022} to bound the truncation error.

    \begin{lemma}\label{lem:sumr}
        Given $\alpha > 1/2$ and  $M \ge 1$, for any $q \in (1/(2\alpha), 1)$, we have
    $$
    \sum_{\veck \in \mathbb{Z}^d \setminus \setA} \frac{1}{r_{\alpha,\vecgamma}^2(\veck)} \le \frac{1}{(\gamma_1 M)^{(1/q-1)/\alpha}} \frac{q}{1-q} \prod_{j=1}^d \left( 1 + 2\gamma_j^{2q} \zeta(2\alpha q) \right)^{1/q}.
    $$
    \end{lemma}

    \begin{theorem}\label{thm:Linf}
    Let $\alpha > 1/2$ and $N$ be a prime number. Let $\vecg$ be constructed by the CBC algorithm (Algorithm \ref{alg:CBCg}), $M = \sup\limits_{1/\alpha < \lambda \le 2} \left( \frac{N}{2} \prod_{j=1}^d (1 + 2\gamma_j^\lambda \zeta(\alpha\lambda))^{-1} \right)^{1/\lambda}$, and $\{\vecy_s\}_{s=1}^{S}$ be constructed by Algorithm \ref{alg:adaptive_shift}. Then for any $\tau \in (0, 1]$, we have 
\begin{align*}
    e_{\infty}^{\text{wor}}(\mathcal{A}, \mathcal{K}_{d, \alpha, \vecgamma}) &\le C_{d, \vecgamma, \tau} (\log N)^{(d-1)/2} N^{-\alpha+1/2+\tau} \\
    &\le C_{d, \vecgamma, \tau} (\log N_{tot})^{(2\alpha-1/2-2\tau)d-1/2} N_{tot}^{-\alpha+1/2+\tau}.
\end{align*}
\end{theorem}

    \begin{proof}    
    Following the strategy of \cite[Theorem 4.2]{cai2025}, we decompose the error into a truncation error and an aliasing error:
$$
\|f - \mathcal{A}(f)\|_{L_\infty} \le \sum_{\veck \in \mathcal{A}_{\alpha,\vecgamma,M}} |\widehat{f}(\veck) - \mathcal{B}_N(f, \veck,\{\vecy_s\},\vecz)| + \sum_{\veck \in \mathbb{Z}^d \setminus \mathcal{A}_{\alpha,\vecgamma,M}} |\widehat{f}(\veck)|. 
$$
For the truncation error, by applying the Cauchy-Schwarz inequality, we have
$$
\begin{aligned}
\sum_{\veck \in \mathbb{Z}^d \setminus \mathcal{A}_{\alpha,\vecgamma,M}} |\widehat{f}(\veck)|
\le \|f\|_{\SpaceKor} \left( \sum_{\veck \in \mathbb{Z}^d \setminus \mathcal{A}_{\alpha,\vecgamma,M}} \frac{1}{(r_{\alpha,\vecgamma}(\veck))^2} \right)^{1/2}.
\end{aligned}
$$

For the aliasing error, let $Z_j := \{\veck \in \mathbb{Z}^d : \veck - \vecl^{(j)} \in \mathcal{L}_N^\perp(\vecg)\}$ denote the set of frequencies aliased with the $j$-th fiber representative $\vecl^{(j)}$. 
Let
$
e_j(\veck) = [B_j^H B_j]^{-1} B_j^H \boldsymbol{a}_{\veck},
$ where the vector $\boldsymbol{a}_{\veck}$ has components $\boldsymbol{a}_{\veck,s} = e^{2\pi \mathrm{i} \veck \cdot \vecy_s}$, $s = 1, 2, \ldots, S$. Following the algebraic derivation in \cite[Proof of Theorem 4.2]{cai2025}, the aliasing error is bounded by
$$
\begin{aligned}
\sum_{\veck \in \setA} |\widehat{f}(\veck) - \mathcal{B}_N(f, \veck,\{\vecy_s\},\vecz)| 
&\le \|f\|_{\SpaceKor}\sqrt{ \sum_{j=1}^J \sum_{\veck \in Z_j \setminus \setA} \frac{\|e_j(\veck)\|_2^2}{r_{\alpha,\vecgamma}^2(\veck)} }.
\end{aligned}
$$

Since each entry of the vector $B_j^H \boldsymbol{a}_{\veck} \in \mathbb{C}^{v_j}$ is a sum of $S$ phase factors, its magnitude is bounded by $S$, which trivially implies $\|B_j^H  \boldsymbol{a}_{\veck} \|_2^2 \le v_jS^2$.
By Lemma \ref{lem:G}, we have $\|(B_j^H B_j)^{-1}\|_2 \le \frac{1}{S(1-t)}$. Thus,
\begin{equation}\label{eq:ej}
    \|e_j(\veck)\|_2^2 \le \|[B_j^H B_j]^{-1}\|_2^2 \|B_j^H  \boldsymbol{a}_{\veck} \|_2^2 \le \frac{1}{S^2(1-t)^2} v_jS^2 = \frac{v_j}{(1-t)^2} \le \frac{R}{(1-t)^2}. 
\end{equation}

Substituting this bound back into the aliasing error summation yields
$$
\begin{aligned}
\sum_{j=1}^J\sum_{\veck \in Z_j \setminus \setA} \frac{\|e_j(\veck)\|_2^2}{r_{\alpha,\vecgamma}^2(\veck)}& = \frac{R}{(1-t)^2}\sum_{j=1}^J \sum_{\veck \in Z_j \setminus \setA} \frac{1}{r_{\alpha,\vecgamma}^2(\veck)} \\&\le \frac{R}{(1-t)^2}\sum_{\veck \in Z^d \setminus \setA} \frac{1}{r_{\alpha,\vecgamma}^2(\veck)} .
\end{aligned}
$$
Finally, combining both error components, we obtain
$$
\begin{aligned}
\|f - \mathcal{A}(f)\|_{L_\infty} &\le \sum_{\veck \in \mathcal{A}_{\alpha,\vecgamma,M}} |\widehat{f}(\veck) - \mathcal{B}_N(f, \veck)| + \sum_{\veck \in \mathbb{Z}^d \setminus \mathcal{A}_{\alpha,\vecgamma,M}} |\widehat{f}(\veck)| \\
& \le \|f\|_{\SpaceKor}\left(1+\frac{\sqrt{R}}{1-t}\right)\left( \sum_{\veck \in \mathbb{Z}^d \setminus \mathcal{A}_{\alpha,\vecgamma,M}} \frac{1}{(r_{\alpha,\vecgamma}(\veck))^2} \right)^{1/2} .
\end{aligned}
$$

By applying Lemma \ref{lem:sumr} to evaluate the truncation tail, substituting the choice of $M$ and using the bound of $N_{tot}$ by \eqref{eq:Ntot}, the asymptotic result follows immediately. This completes the proof.
\end{proof}

We now consider the randomized $L_2$ error
$$
e_2^{\text{ran}}(\mathcal{A}_{\bDelta}, \mathcal{K}_{d, \alpha, \vecgamma}) = \sup_{\substack{f \in \mathcal{K}_{d, \alpha, \vecgamma} \\ \|f\|_{\mathcal{K}_{d, \alpha, \vecgamma}} \le 1}} \sqrt{\mathbb{E}_{\bDelta} \|f - \mathcal{A}_{\bDelta}(f)\|_2^2},
$$ 
where $\mathcal{A}_{\bDelta}(f)$ is given by \eqref{eq:appf}.

\begin{theorem}\label{thm:L2}
    Let $\alpha > 1/2$ and $N$ be a prime number. Let $\vecg$ be constructed by the CBC algorithm (Algorithm \ref{alg:CBCg}), $M = \sup_{1/\alpha < \lambda \le 2} \left( \frac{N}{2} \prod_{j=1}^d (1 + 2\gamma_j^\lambda \zeta(\alpha\lambda))^{-1} \right)^{1/\lambda}$, and $\{\vecy_s\}_{s=1}^{S}$ be constructed by Algorithm \ref{alg:adaptive_shift}. Then for any $\tau \in (0, 1]$, we have 
\begin{equation}
    e_2^{\text{ran}}(\mathcal{A}_{\bDelta}, \mathcal{K}_{d, \alpha, \vecgamma}) \le C_{d, \vecgamma, \tau} (\log N)^{(d-1)/2} N^{-\alpha+\tau} \le C_{d, \vecgamma, \tau}(\log N_{tot})^{(2\alpha-2\tau+1/2)d-1/2} N_{tot}^{-\alpha+\tau}.
\end{equation}
\end{theorem}

\begin{proof}
    Following the strategy of \cite[Theorem 4.3]{cai2025} and utilizing the bounds established in Theorem \ref{thm:Linf}, we decompose the expected squared error as follows: 
$$
\mathbb{E}_{\bDelta} [\|f - \mathcal{A}_{\bDelta}(f)\|_2^2] = \sum_{\veck \in \setA} \mathbb{E}_{\bDelta} [|\widehat{f}(\veck) - \mathcal{B}_N(f, \veck, \{\vecy_s\}, \bDelta)|^2] + \sum_{\veck \in \mathbb{Z}^d \setminus \setA} |\widehat{f}(\veck)|^2.
$$
For the truncation error, we have
\begin{align}
\sum_{\veck \in \mathbb{Z}^d \setminus \setA} |\widehat{f}(\veck)|^2 & = \sum_{\veck \in \mathbb{Z}^d \setminus \setA} |\widehat{f}(\veck)|^2 \frac{r_{\alpha}^2(\veck)}{r_{\alpha}^2(\veck)} \label{L2_truncation} \\ & \le \frac{1}{M^2} \sum_{\veck \in \mathbb{Z}^d \setminus \setA} |\widehat{f}(\veck)|^2 r_{\alpha}^2(\veck) \nonumber \le \frac{\|f\|_{d,\alpha,\vecgamma}^2}{M^2}. \nonumber
\end{align}
For the aliasing error, applying the derivation from \cite[Theorem 4.3]{cai2025} and substituting the uniform bound \eqref{eq:ej} for $\|e_j(\veck)\|_2^2$, it follows that
$$
\begin{aligned}
\sum_{\veck \in \setA} \mathbb{E}_{\bDelta} \left[ |\widehat{f}(\veck) - \mathcal{B}_N(f, \veck, \{\vecy_s\}, \bDelta)|^2 \right] &\le \max_{j, \veck}\|e_j(k)\|_2^2 \sum_{j=1}^J \sum_{\substack{\veck - \vecl^{(j)} \in \mathcal{L}_N^\perp(\vecg) \\ \veck \notin \Gamma_{\alpha,\vecgamma,N}(\vecl^{(j)}; \vecg)}} |\widehat{f}(\veck)|^2 \\
&\le \frac{R}{(1-t)^2} \sum_{\veck \in \mathbb{Z}^d \setminus \setA} |\widehat{f}(\veck)|^2 \le \frac{\|f\|_{d,\alpha,\vecgamma}^2}{M^2}\frac{R}{(1-t)^2}.
\end{aligned}
$$
Combining these bounds, the total variance is bounded by
$$
\mathbb{E}_{\bDelta} [\|f - \mathcal{A}_{\bDelta}(f)\|_2^2] \leq \|f\|_{\mathcal{K}_{d,\alpha,\vecgamma}}^2\frac{1 + R/(1-t)^2}{M^2}.
$$
Taking the supremum over the unit ball in the Korobov space, we obtain the worst-case randomized error:
$$
e_2^{\mathrm{ran}}(\mathcal{A}_{\bDelta}, \mathcal{K}_{d,\alpha,\vecgamma}) \leq \frac{\sqrt{1 + R/(1-t)^2}}{M} = \mathcal{O}\left(N^{-\alpha+\tau}(\log N)^{(d-1)/2}\right).
$$
The asymptotic scaling with respect to $N_{tot}$ follows directly by \eqref{eq:Ntot}. This completes the proof.
\end{proof}

\begin{remark}
Our simplified uniform shift strategy yields a fundamental improvement. As derived in Theorem 2, we bound $\|B^H a_k\|_2 \le v^{1/2}S$, leading to $\|e_j(k)\|_2 \le \frac{\sqrt{R}}{1-t} = \mathcal{O}(R^{1/2})$. This strict reduction of the penalty term from $\mathcal{O}(R^{3/2})$ to $\mathcal{O}(R^{1/2})$ explicitly quantifies the theoretical advantage of the proposed framework in Section \ref{sec:simp}, which translates to a reduction by a factor of $\mathcal{O}((\log N)^{d-1})$. Analogous asymptotic improvements apply to the $L_2$ error.
\end{remark}

\section{Approximation on Non-periodic functions and Applications to Non-periodic Neumann Problems}
\label{sec:cosine_application}

The theoretical framework established in the preceding sections fundamentally relies on the periodicity of the functions in the Korobov space $\mathcal{K}_{d,\alpha,\vecgamma}$. However, in practical physical scenarios, such as the Poisson equation with Neumann boundary conditions discussed in Section~\ref{sec_Neumann}, the scalar fields are intrinsically non-periodic. In this section, we demonstrate that the proposed deterministic multiple-shift lattice algorithm can be structurally extended to the non-periodic half-period cosine space via a tent transformation. In Section~\ref{sec_Neumann} we show how this can then be used as a generic meshless spectral solver for high-dimensional boundary value problems.

\subsection{The Half-Period Cosine Space and Tent Transformation}

Let $\mathbb{Z}_{+} := \{0, 1, 2, \ldots\}$. For a frequency vector $\veck \in \mathbb{Z}_{+}^d$, the tensor product half-period cosine basis functions are defined as
\begin{equation}
    \phi_{\veck}(x) := \prod_{j=1}^{d} \phi_{k_j}(x_j) = (\sqrt{2})^{|\veck|_0} \prod_{j=1}^{d} \cos(\pi k_j x_j), \quad x \in [0,1]^d,
\end{equation}
where $|\veck|_0$ denotes the number of non-zero components of $\veck$. 

\begin{definition}
Let $d \in \mathbb{N}$, $\alpha > 1/2$ and $\vecgamma = (\gamma_1, \gamma_2, \dots)$ be a sequence of positive real numbers. The weighted half-period cosine space $\mathcal{C}_{d,\alpha,\vecgamma}$ is defined as the reproducing kernel Hilbert space of non-periodic functions $f: [0,1]^d \to \mathbb{R}$ represented by uniformly convergent cosine series $f(x) = \sum_{\veck \in \mathbb{Z}_{+}^d} \widetilde{f}(\veck) \phi_k(x)$, equipped with the norm
\begin{equation}
    \|f\|_{\mathcal{C}_{d,\alpha,\vecgamma}}^2 := \sum_{\veck \in \mathbb{Z}_{+}^d} |\widetilde{f}(\veck)|^2 r_{\alpha,\vecgamma}^2(\veck) < \infty,
\end{equation}
where $\widetilde{f}(\veck) = \int_{[0,1]^d} f(x) \phi_k(x) dx$ are the cosine coefficients.
\end{definition}

For the critical case $\alpha = 1$, the half-period cosine space is isomorphic to the (non-periodic) unanchored Sobolev space
of dominated mixed smoothness 1 \cite{dick2014lattice}. Specifically, $\mathcal{C}_{d,1,\vecgamma}$ is the reproducing kernel Hilbert space of functions defined on $[0,1]^d$ with square-integrable mixed first derivatives, equipped with the analytical norm
\begin{equation}
    \|f\|_{\mathcal{C}_{d,1,\vecgamma}}^2 := \sum_{\boldsymbol{u} \subseteq \{1, \ldots, d\}} \prod_{j\in \boldsymbol{u}}\frac{1}{\gamma_j} \int_{[0,1]^{|\boldsymbol{u}|}} \left( \int_{[0,1]^{d-|\boldsymbol{u}|}} \frac{\partial^{|\boldsymbol{u}|} f}{\partial \vecx_{\boldsymbol{u}}}(\vecx) \mathrm{d}\vecx_{-\boldsymbol{u}} \right)^2 \mathrm{d}\vecx_{\boldsymbol{u}}, \nonumber
\end{equation}
where $\frac{\partial^{|\boldsymbol{u}|} f}{\partial \vecx_{\boldsymbol{u}}}$ denotes the mixed first derivatives of $f$ with respect to the variables $x_j$ with $j \in \boldsymbol{u}$. 

The corresponding truncation index set is denoted by $\setA^{+} := \{\veck \in \mathbb{Z}_{+}^d : r_{\alpha,\vecgamma}(\veck) < M\}$.

To apply the lattice rule, we employ the tent transformation $\psi: [0,1] \to [0,1]$ given by $\psi(z) := 1 - |2z - 1|$, applied component-wise to map a point $z \in [0,1]^d$ to $x = \psi(z)$. It is a standard result that if $f \in \mathcal{C}_{d,\alpha,\vecgamma}$, the transformed function $g(z) := f(\psi(z))$ is 1-periodic and even with respect to $z_j=1/2$, and belongs to the Korobov space $\mathcal{K}_{d,\alpha,\vecgamma}$ with the isometric isomorphism $\|f\|_{\mathcal{C}_{d,\alpha,\vecgamma}} = \|g\|_{\mathcal{K}_{d,\alpha,\vecgamma}}$ \cite{kuo2021function}. 

The Fourier coefficients $\widehat{g}(\vech)$ for $\vech \in \mathbb{Z}^d$ are geometrically linked to the cosine coefficients $\widetilde{f}(\veck)$ via $\widehat{g}(\vech) = 2^{-|\vech|_0/2} \widetilde{f}(|\vech|)$. Consequently, we recover the cosine coefficients using the approximated Fourier coefficients $\widehat{c}_{\vech} \approx \widehat{g}(\vech)$ generated by our deterministic multiple-shift algorithm via the following identity \cite{suryanarayana2016}:
\begin{equation*}
    \widetilde{f}(\veck) \approx \widetilde{c}_{\veck} := \frac{1}{2^{|\veck|_0/2}} \sum_{|\vech|=\veck} \widehat{c}_{\vech}.
\end{equation*}

From \eqref{eq_zero_fiber} we obtain that the fiber containing $\vecz$ only has this one element. Hence we have the approximation $\widetilde{f}(\vecz) \approx \widetilde{c}_{\vecz} = \widehat{c}_{\vecz}$.

Now consider $\vech$ beloning to a fiber which does not contain $\vecz$. Since the target function $g(\bz)$ is strictly real-valued and the frequency truncation set is symmetric, the least-squares system for the negated frequencies $-\vech$ employs the complex conjugate of the original basis matrix. This structurally guarantees that the linear weights derived from the pseudo-inverse $(B^H B)^{-1} B^H$ are identically conjugated. Consequently, the reconstructed Fourier coefficients naturally satisfy the conjugate symmetry $\widehat{c}_{-\vech} = \overline{\widehat{c}_{\vech}}$. Thus, the recovered cosine coefficients $\widetilde{c}_{\veck}$ strictly lie in $\mathbb{R}$, yielding a real-valued cosine approximation 
\begin{equation}\label{expansion_tildef}
\mathcal{A}^{\mathrm{cos}}(f) = \widetilde{f}(\vecx) := \sum_{\veck \in \setA^{+}} \widetilde{c}_{\veck} \phi_{\veck}(\vecx).
\end{equation}

\subsection{Projection Equivalence via the Averaging Operator}

To strictly bound the approximation error in $\mathcal{C}_{d,\alpha,\vecgamma}$, we define the averaging operator $\mathcal{M}$ over the sign group $\{\pm 1\}^d$ for any function $u$ on $\mathbb{R}^d$:
\begin{equation*}
    \mathcal{M}[u](\bz) := \frac{1}{2^d} \sum_{\sigma \in \{\pm 1\}^d} u(\sigma(\bz)),
\end{equation*}
where $\sigma(\bz) = (\sigma_1 z_1, \ldots, \sigma_d z_d)$. The operator $\mathcal{M}$ is linear. By the triangle inequality and Jensen's inequality, it satisfies $\|\mathcal{M}[u]\|_\infty \le \|u\|_\infty$ and $\|\mathcal{M}[u]\|_2 \le \|u\|_2$. Since $\widehat{g}(\sigma(\vech)) = \widehat{g}(\vech)$, it follows that $\mathcal{M}[g] = g$. Furthermore, the algebraic construction of $\widetilde{c}_{\veck}$ ensures that the spatial approximations are equivalent under this projection.

\begin{lemma}[Projection Equivalence]
Let $\widetilde{g}(\bz) = \sum_{\vech \in \setA} \widehat{c}_{\vech} \exp(2\pi \mathrm{i} \vech \cdot \bz)$ be the approximation in $\mathcal{K}_{d,\alpha,\vecgamma}$ constructed via the deterministic multiple-shift algorithm. Then, the corresponding cosine approximation $\widetilde{f}(\vecx)$ satisfies:
\begin{equation*}
    \widetilde{f}(\psi(\bz)) = \mathcal{M}[\widetilde{g}](\bz).
\end{equation*}
\end{lemma}

\begin{proof}
By \cite[Lemma 1]{cools2016tent}, we have $\phi_{\veck}(\psi(\bz)) = \frac{(\sqrt{2})^{|\veck|_0}}{2^d} \sum_{\sigma \in \{\pm 1\}^d} \exp(2\pi \mathrm{i} \sigma(\veck) \cdot \bz)$. Substituting this into the expansion \eqref{expansion_tildef} of $\widetilde{f}$ yields:
\begin{align*}
    \widetilde{f}(\psi(\bz)) &= \sum_{\veck \in \setA^{+}} \left[ \frac{1}{2^{|\veck|_0/2}} \sum_{|\vech|=\veck} \widehat{c}_{\vech} \right] \left[ \frac{(\sqrt{2})^{|\veck|_0}}{2^d} \sum_{\sigma \in \{\pm 1\}^d} \exp(2\pi \mathrm{i} \sigma(\veck) \cdot \bz) \right] \\
    &= \sum_{\veck \in \setA^{+}} \sum_{|\vech|=\veck} \widehat{c}_{\vech} \frac{1}{2^d} \sum_{\sigma \in \{\pm 1\}^d} \exp(2\pi \mathrm{i} \sigma(\veck) \cdot \bz) \\ & = \frac{1}{2^d} \sum_{\sigma \in \{\pm 1\}^d} \sum_{\veck \in \setA^{+}} \sum_{|\vech|=\veck} \widehat{c}_{\vech} \exp(2\pi \mathrm{i} \veck \cdot \sigma(\bz)).
\end{align*}
For any fixed $\veck$ and a specific $\vech$ satisfying $|\vech|=\veck$, the inner product satisfies $\veck \cdot \sigma(\bz) = \vech \cdot \sigma'(\sigma(\bz))$ for some sign vector $\sigma' \in \{\pm 1\}^d$ such that $\vech = \sigma'(\veck)$. Since summing over all $\sigma \in \{\pm 1\}^d$ is invariant under component-wise sign flips of the frequency vector, we have $\sum_{\sigma} \exp(2\pi \mathrm{i} \veck \cdot \sigma(\bz)) = \sum_{\sigma} \exp(2\pi \mathrm{i} \vech \cdot \sigma(\bz))$. This allows the sums to be seamlessly rearranged over the entire symmetric index space $\setA$:
\begin{equation*}
    \widetilde{f}(\psi(\bz)) = \frac{1}{2^d} \sum_{\sigma \in \{\pm 1\}^d} \sum_{\vech \in \setA} \widehat{c}_{\vech} \exp(2\pi \mathrm{i} \vech \cdot \sigma(\bz)) = \frac{1}{2^d} \sum_{\sigma \in \{\pm 1\}^d} \widetilde{g}(\sigma(\bz)) = \mathcal{M}[\widetilde{g}](\bz).
\end{equation*}
This completes the proof.
\end{proof}

\subsection{Error Bounds in Cosine Space}
With the projection equivalence established, the approximation error in the non-periodic space is rigorously bounded by the error in the intermediate periodic space. Note that the logarithmic penalty factor associated with the maximum fiber length $R$ persists in this projection.

\begin{theorem}[Optimal Convergence in $\mathcal{C}_{d,\alpha,\vecgamma}$]\label{thm:errf}
Let $\widetilde{f}$ be the reconstructed function in the half-period cosine space using the deterministic multiple-shift strategy. For any $\tau \in (0, 1]$, the approximation satisfies the optimal worst-case and randomized asymptotic bounds corresponding to the Korobov algorithm:
\begin{align*}
    e_{\infty}^{\mathrm{wor}}(\mathcal{A}^{\mathrm{cos}}, \mathcal{C}_{d,\alpha,\vecgamma}) &\le C_{d, \vecgamma, \tau} (\log N)^{(d-1)/2} N^{-\alpha+1/2+\tau} \nonumber \\
    &\le C_{d, \vecgamma, \tau} (\log N_{tot})^{(2\alpha-2\tau-1/2)d-1/2} N_{tot}^{-\alpha+1/2+\tau},
\end{align*}
and
\begin{align*}
    e_2^{\mathrm{ran}}(\mathcal{A}_{\bDelta}^{\mathrm{cos}}, \mathcal{C}_{d,\alpha,\vecgamma}) &\le C_{d, \vecgamma, \tau} (\log N)^{(d-1)/2} N^{-\alpha+\tau} \nonumber \\
    &\le C_{d, \vecgamma, \tau}(\log N_{tot})^{(2\alpha-2\tau+1/2)d-1/2} N_{tot}^{-\alpha+\tau}.
\end{align*}
\end{theorem}

\begin{proof} 
For a given $\vecx \in [0,1]^d$ let $\bz \in [0,1]^d$ be such that $\vecx = \psi(\bz)$. Let $E(\bz) = g(\bz) - \widetilde{g}(\bz)$ denote the approximation error in the Korobov space.  Since the target function satisfies $\mathcal{M}[g] = g$ and the approximation satisfies $\widetilde{f}(\psi(\bz)) = \mathcal{M}[\widetilde{g}](\bz)$, the pointwise error is strictly bounded by:
\begin{equation*}
    |f(\vecx) - \widetilde{f}(\vecx)| = |f(\psi(\bz)) - \widetilde{f}(\psi(\bz))| = |g(\bz) - \mathcal{M}[\widetilde{g}](\bz)| = |\mathcal{M}[g - \widetilde{g}](\bz)| \le \|\mathcal{M}[E]\|_{\infty} \le \|E\|_{\infty}.
\end{equation*}
Taking the supremum over the unit ball $\|f\|_{\mathcal{C}_{d,\alpha,\vecgamma}} \le 1$ and applying the isometric isomorphism $\|f\|_{\mathcal{C}_{d,\alpha,\vecgamma}} = \|g\|_{\mathcal{K}_{d,\alpha,\vecgamma}}$ directly yields the worst-case bound. For the randomized setting, the mean square error follows identically from the operator property $\|\mathcal{M}[E]\|_2 \le \|E\|_2$. Thus, the bounds with respect to both the base lattice size $N$ and the total sampling cost $N_{tot}$ are strictly anchored to the algebraic rates established for the Korobov space in Theorems \ref{thm:Linf} and \ref{thm:L2}, preserving the dimension-dependent logarithmic penalties. 
\end{proof}

\subsection{Spectral Method for Non-periodic Neumann Problems}\label{sec_Neumann}
The strict error bounds established above directly authorize the use of $\mathcal{A}_{\Delta}^{\mathrm{cos}}$ as a generic meshless spectral solver for physical boundary value problems. Non-homogeneous Neumann boundary conditions can be analyzed similarly \cite{suryanarayana2016}. Consider the Poisson problem with homogeneous Neumann boundary conditions:
\begin{equation*}
    \nabla^2 u(x) = f(x), \quad x \in (0,1)^d,
\end{equation*}
subject to $\frac{\partial u}{\partial n} = 0$ on $\partial(0,1)^d$. By the spectral decomposition of the negative Laplacian, the exact solution $u(x)$ is given by:
\begin{equation}
    u(x) = \widehat{u}(0) - \sum_{k \in \mathbb{Z}_{+}^d \setminus \{0\}} \frac{\widetilde{f}(k)}{\lambda_k} \phi_k(x),
\end{equation}
where $\lambda_k = \pi^2 \sum_{j=1}^d k_j^2$ are the eigenvalues, and the mean value $\widehat{u}(0) = \int_{[0,1]^d} u(x) dx$ is assumed known. The numerical solution is constructed by substituting the source function coefficients with our lattice approximation $\widetilde{c}_k$:
\begin{equation}
    \widetilde{u}(x) = \widehat{u}(0) - \sum_{k \in \setA^{+} \setminus \{0\}} \frac{\widetilde{c}_k}{\lambda_k} \phi_k(x).
\end{equation}

\begin{remark}
It is well-known that the solution to the Poisson equation with Neumann boundary conditions is unique only up to an additive constant. To fix the solution, we assume that $\widehat{u}(0)$ is known. This assumption is physically reasonable and consistent with real-world applications. In many physical scenarios modeled by Neumann problems---such as closed-system diffusion or incompressible flow---the mean value $\widehat{u}(0)$ is not arbitrary but is strictly determined by global conservation laws (e.g., conservation of total mass or global pressure constraints) \cite{strang2007}. Therefore, treating $\widehat{u}(0)$ as a given prior is a standard practice in solving such boundary value problems.
\end{remark}

\begin{theorem}[Convergence of the PDE Solution]\label{thm:erru}
Let the source function $f \in \mathcal{C}_{d,\alpha,\vecgamma}$ and assume $\widehat{u}(0)$ is given. For any $\tau \in (0, 1]$, the root-mean-square error of the approximate PDE solution $\widetilde{u}$ achieves the optimal algebraic decay inherited from the source approximation:
\begin{equation*}
\begin{aligned}
    \sqrt{\mathbb{E}_{\Delta}\|u - \widetilde{u}\|_2^2} &\le \frac{C_{d, \vecgamma, \tau}}{\pi^2} (\log N)^{(d-1)/2} N^{-\alpha +\tau} \\
    &\le \frac{C_{d, \vecgamma, \tau}}{\pi^2}(\log N_{tot})^{(2\alpha-2\tau+1/2)d-1/2} N_{tot}^{-\alpha+\tau}.
\end{aligned}
\end{equation*}
An analogous convergence bound can be straightforwardly established for the worst-case $L_\infty$ error using the corresponding bounds on $f$.
\end{theorem}

\begin{proof}
By Parseval's identity for the orthonormal cosine basis, the squared $L_2$ error is isolated to the non-zero frequency components:
\begin{equation*}
    \mathbb{E}_{\Delta} \|u - \widetilde{u}\|_2^2 = \mathbb{E}_{\Delta} \sum_{k \neq 0} \frac{|\widetilde{f}(k) - \widetilde{c}_k|^2}{\lambda_k^2}.
\end{equation*}

For all $k \neq 0$, the Laplacian eigenvalues satisfy $\lambda_k \ge \pi^2$. Extracting this physical smoothing factor yields $1/\lambda_k^2 \le 1/\pi^4$. Consequently:
\begin{equation*}
   \mathbb{E}_{\Delta} \|u - \widetilde{u}\|_2^2 \le \frac{1}{\pi^4} \mathbb{E}_{\Delta} \sum_{k \neq 0} |\widetilde{f}(k) - \widetilde{c}_k|^2 \le \frac{1}{\pi^4} \mathbb{E}_{\Delta} \|f - \widetilde{f}\|_2^2.
\end{equation*}
Substituting the bound $e_2^{\mathrm{ran}}(\mathcal{A}_{\Delta}^{\mathrm{cos}}, \mathcal{C}_{d,\alpha,\vecgamma})$ from Theorem 4 and taking the square root completes the proof.
\end{proof}

\section{Numerical Experiments}\label{sec:numer}
In this section, we present numerical experiments to systematically evaluate the empirical performance of the proposed deterministic adaptive multiple-shift algorithm. 

\subsection{Periodic Test Function}
We evaluate the algorithm using a highly smooth periodic test function commonly analyzed in the literature \cite{cai2024l_2,pan2025universal,pan2025}:
$$
f_1(\vecx) = \prod_{j=1}^d \left( x_j - \frac{1}{2} \right)^2 \sin(2\pi x_j - \pi).
$$
We consider dimensions $d=2$ and $d=4$ to observe the dimensional influence. This test function belongs to the space $\mathcal{K}_{d,5/2-\epsilon,\vecgamma}$ for arbitrarily small $\epsilon > 0$. Accordingly, we set the smoothness parameter $\alpha = 5/2$. The weights are fixed at $\gamma_j = 2^{(1-j)/10}$ and the construction parameter $t = 0.95$.

We first investigate the empirical sampling cost of the adaptive strategy. As shown in Figure \ref{fig:S_growth_analysis}, the required number of shifts $S$ exhibits an empirical growth well-approximated by a quadratic polynomial in $\log_2 N$. The observable increase in the quadratic coefficient from $d=2$ to $d=4$ reflects the internal algorithmic transition from the polynomial strategy to the single-lattice strategy as the geometric capacity constraints tighten.

With this empirical sampling cost established, we analyze the approximation errors for $f_1$ in Figure \ref{fig:convergence_func1}. As demonstrated in the left column, the empirical errors plotted against the base lattice size $N$ perfectly align with the theoretical optimal rates for $d=2$: specifically $\mathcal{O}(N^{-\alpha})$ for the $L_2$ error and $\mathcal{O}(N^{-\alpha+1/2})$ for the $L_\infty$ error. For $d=4$, the convergence curves initially exhibit a pre-asymptotic phenomenon driven by the dimension-dependent logarithmic penalty $\mathcal{O}((\log N)^{(d-1)/2})$. However, as $N$ increases, the asymptotic tail strictly accelerates toward the optimal theoretical slope. 

When measuring the error against the total sampling cost $N_{tot} = N \times S$ (Figure \ref{fig:convergence_func1}, right column), the $\mathcal{O}((\log N)^2)$ inflation of $S$ introduces a a quantifiable pre-asymptotic drag. For $d=2$, the observed rate temporarily reduces to approximately $\mathcal{O}(N_{tot}^{-3\alpha/4})$, which is consistent with empirical findings in recent randomized frameworks \cite{pan2025universal,pan2025}. For $d=4$, compounded by the larger quadratic coefficient, this drag reduces the apparent rate further to $\mathcal{O}(N_{tot}^{-\alpha/2})$. Nevertheless, because the deterministic shifts $\{\vecy_s\}_{s=1}^S$ are exclusively designed to algebraically guarantee system invertibility without corrupting the fundamental spatial quadrature accuracy.  As clearly observed in the left column, the convergence curve for $d=4$ with respect to $N$ already exhibits an accelerated decay in its asymptotic tail. This empirical acceleration with respect to $N$ strongly indicates that the $-\alpha/2$ degradation currently observed against $N_{tot}$ is strictly a temporary pre-asymptotic phenomenon. It validates our theoretical expectation: as $N \to \infty$ and the logarithmic penalty from $S$ is asymptotically absorbed, the convergence curve with respect to the total cost $N_{tot}$ will recover the optimal $-\alpha$ rate.

\vspace{-3mm}
\begin{figure}[htbp]
        \centering
        \includegraphics[width=0.5\textwidth]{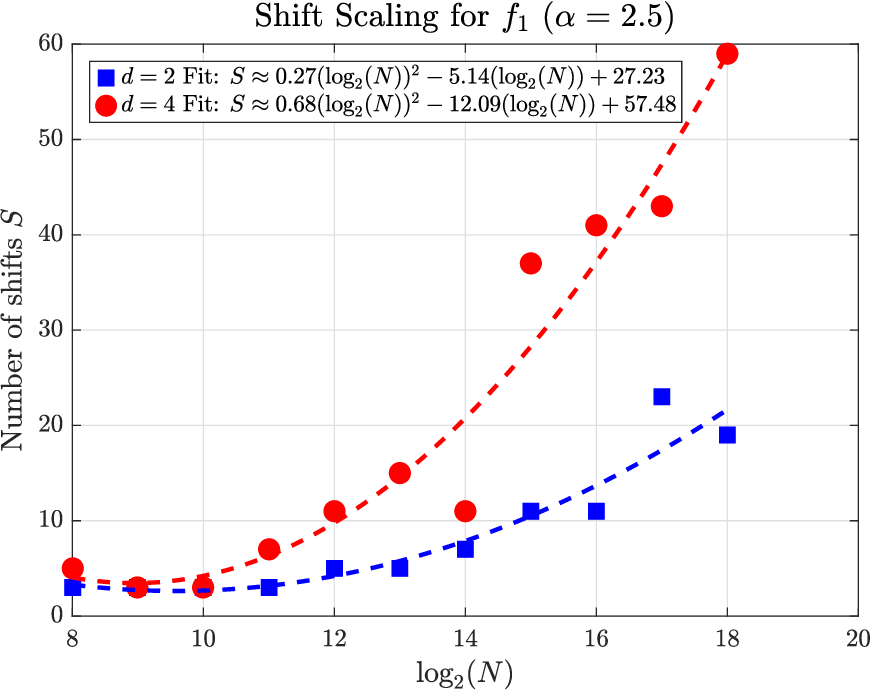}
    \caption{Empirical growth of the required number of shifts $S$ with respect to $\log_2(N)$.}
    \label{fig:S_growth_analysis}
    \vspace{-3mm}
\end{figure}

\vspace{-3mm}
\begin{figure}[htbp]
    \centering
    \begin{subfigure}{0.47\textwidth}
        \centering
        \includegraphics[width=\textwidth]{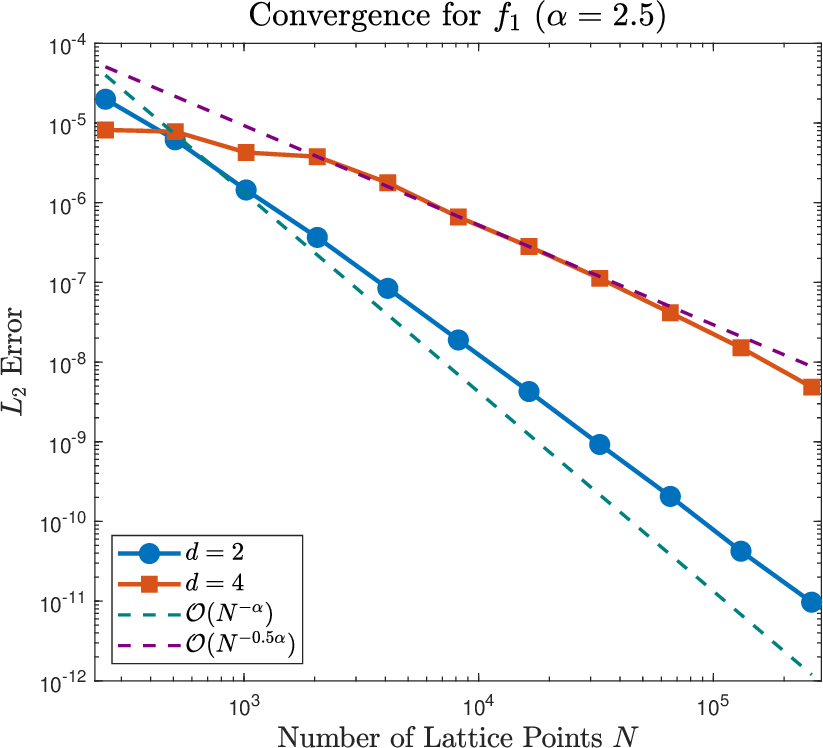} 
        \caption{$L_2$ error w.r.t base lattice size $N$}
        \label{subfig:f1_L2_N}
    \end{subfigure}
    \hfill
    \begin{subfigure}{0.47\textwidth}
        \centering
        \includegraphics[width=\textwidth]{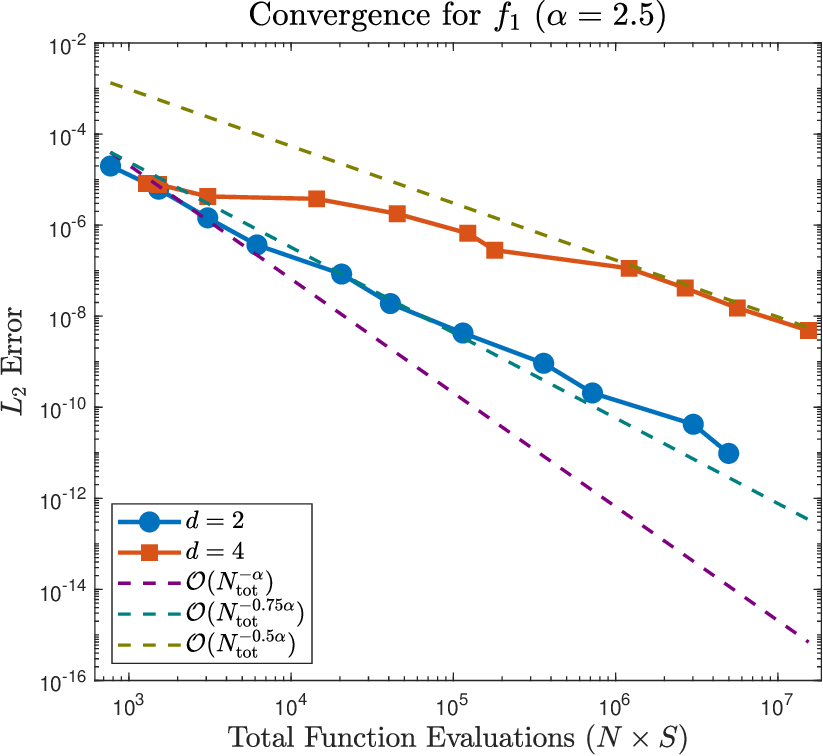}
        \caption{$L_2$ error w.r.t total cost $N_{tot}$}
        \label{subfig:f1_L2_Ntot}
    \end{subfigure}
    
    \vspace{0.4cm} 
    
    \begin{subfigure}{0.47\textwidth}
        \centering
        \includegraphics[width=\textwidth]{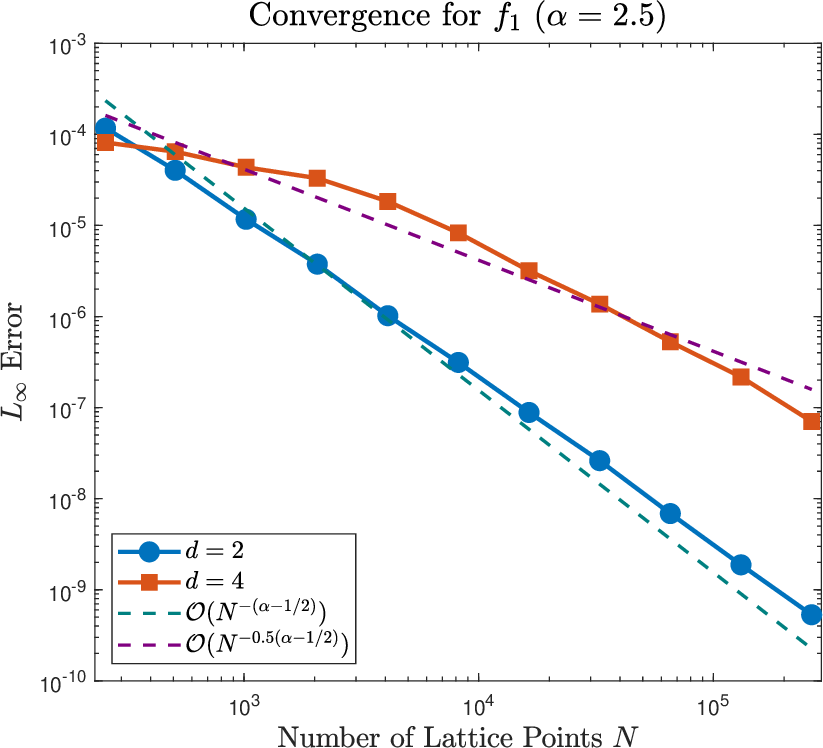}
        \caption{$L_\infty$ error w.r.t base lattice size $N$}
        \label{subfig:f1_Linf_N}
    \end{subfigure}
    \hfill
    \begin{subfigure}{0.47\textwidth}
        \centering
        \includegraphics[width=\textwidth]{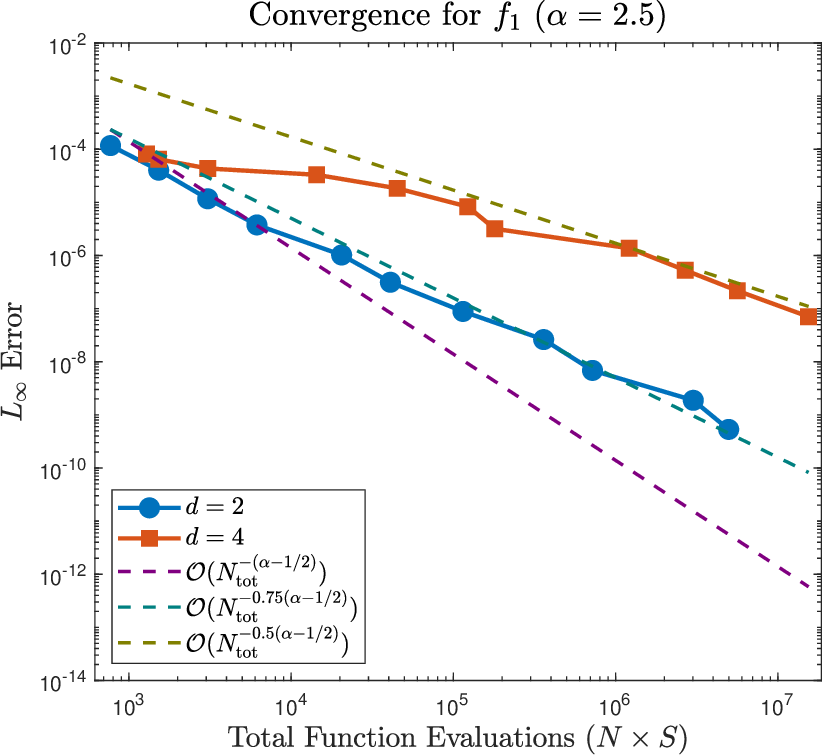}
        \caption{$L_\infty$ error w.r.t total cost $N_{tot}$}
        \label{subfig:f1_Linf_Ntot}
    \end{subfigure}
    
    \caption{The $L_2$ and $L_{\infty}$ errors for the periodic test function $f_1$ across dimensions $d=2$ and $d=4$.}
    \label{fig:convergence_func1}
    \end{figure}
    \vspace{-3mm}

\subsection{Poisson partial differential equation}
To demonstrate the practical utility of the deterministic multiple-shift lattice algorithm in the non-periodic setting, we consider the Poisson problem with homogeneous Neumann boundary conditions. We employ the following highly oscillatory forcing function:
\begin{equation}
    \nabla^2 u(\vecx) = \sum_{j=1}^d \gamma_j(12x_j^2 - 12x_j + 2) \prod_{\substack{i=1 \\ i \neq j}}^d \left( \frac{1}{630} + \gamma_i \left( x_i^2(1 - x_i)^2 - \frac{1}{630} \right) \right) := f(\vecx),
\end{equation}
which is a multidimensional benchmark problem investigated in \cite{suryanarayana2016}. The exact solution for this physical system is known analytically as
\begin{equation}
    u(\vecx) = \prod_{j=1}^d \left( \frac{1}{630} + \gamma_j \left( x_j^2(1 - x_j)^2 - \frac{1}{630} \right) \right),
\end{equation}
subject to the global conservation constraint $\widehat{u}(\vecz) = \prod_{i=1}^d((1 + 20 \gamma_i)/ 630)$. Through the Taylor expansion of the tent-transformed polynomial terms, it can be verified that the source function belongs to the half-period cosine space, $f \in \mathcal{C}_{d,3/2-\epsilon,\vecgamma}$ for arbitrarily small $\epsilon > 0$. Accordingly, we set $\alpha = 1.5$ and $\gamma_j = 2^{(1-j)/10}$. Because the absolute $L_2$ norms of the analytical exact fields $f$ and $u$ are exceptionally small, reporting absolute errors would artificially inflate the perceived accuracy. To ensure a strict and numerically meaningful evaluation, we follow the methodology in \cite{suryanarayana2016} and explicitly report the \textit{relative} $L_2$ errors, i.e., $\|f - \tilde{f}\|_2 / \|f\|_2$ and $\|u - \tilde{u}\|_2 / \|u\|_2$, across dimensions $d = 2$ and $d=4$.

We first verify the empirical sampling cost of the adaptive deterministic shift strategy in the Cosine space. As shown in Figure \ref{fig:pde_shifts}, the required number of shifts $S$ maintains a strictly controlled growth well-approximated by a quadratic polynomial in $\log_2 N$. 

Subsequently, Figure \ref{fig:convergence_pde} presents the relative $L_2$ convergence for both the approximated source function $f$ (top row) and the final PDE solution $u$ (bottom row). The left column isolates the theoretical spatial resolution, strictly bounded by the base lattice size $N$. As established in Theorems \ref{thm:errf} and \ref{thm:erru}, the empirical errors for both $f$ and $u$ perfectly align with their respective optimal algebraic decay rates of $\mathcal{O}(N^{-\alpha})$.

The right column illustrates the practical computational efficiency measured against the total sampling cost $N_{tot} = N \times S$. The curves inevitably exhibit the pre-asymptotic drag inherent to the dimension-dependent shift penalty $\mathcal{O}((\log N)^{2})$. While this temporarily reduces the apparent decay slope for $d=4$, the corresponding spatial curves (left column) already demonstrate a clear downward acceleration, validating that the optimal asymptotic rate is ultimately recovered.

Crucially, across all parameter configurations, the relative error of the PDE solution $u$ remains structurally lower in magnitude than that of the forcing function $f$. This numerical behavior explicitly captures the physical smoothing effect induced by the inverse Laplacian operator, where the squared inverse eigenvalues $1/\lambda_k^2 \le 1/\pi^4$ strictly and uniformly damp the high-frequency oscillatory residuals.

\begin{figure}[htbp]
        \centering
        \includegraphics[width=0.5\textwidth]{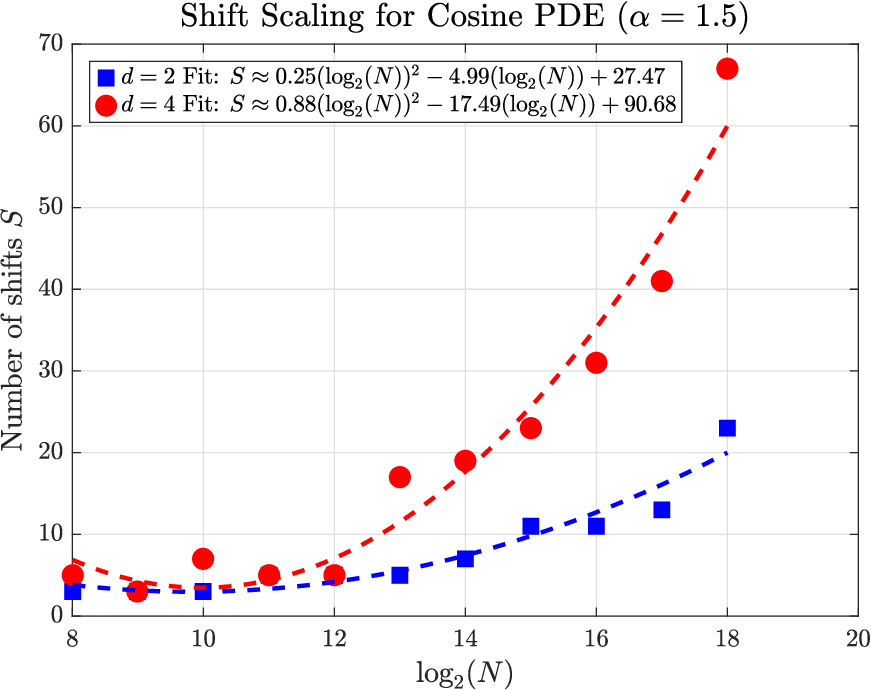}
    \caption{The required number of shifts $S$ for the meshless PDE solver across dimensions $d=2$ and $d=4$.}
    \label{fig:pde_shifts}
\end{figure}

\begin{figure}[htbp]
    \centering
    \begin{subfigure}{0.48\textwidth}
        \centering
        \includegraphics[width=\textwidth]{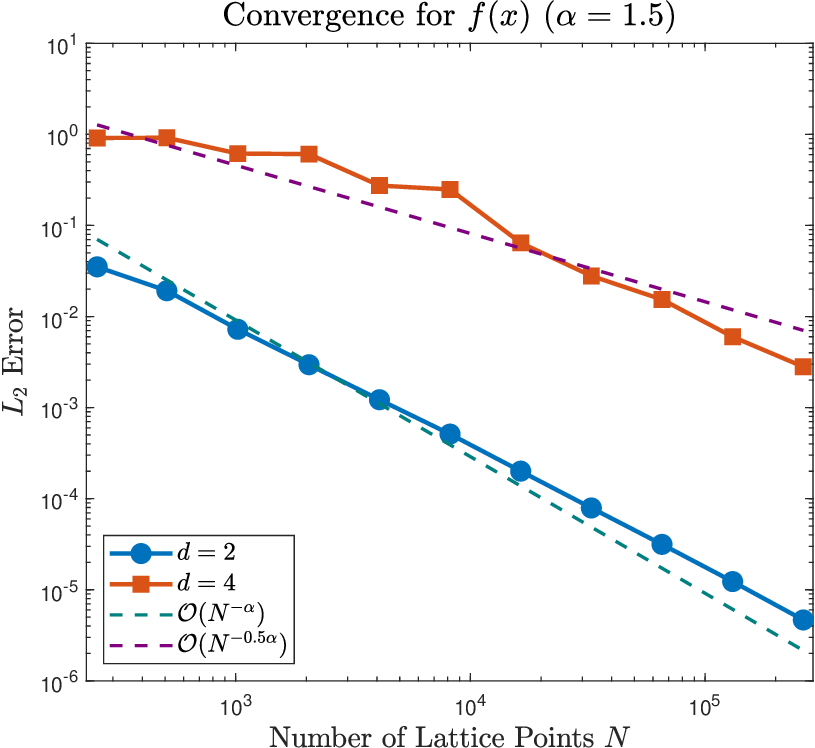} 
        \caption{$L_2$ error of $f$ w.r.t base lattice size $N$}
        \label{subfig:pde_f_L2_N}
    \end{subfigure}
    \hfill
    \begin{subfigure}{0.48\textwidth}
        \centering
        \includegraphics[width=\textwidth]{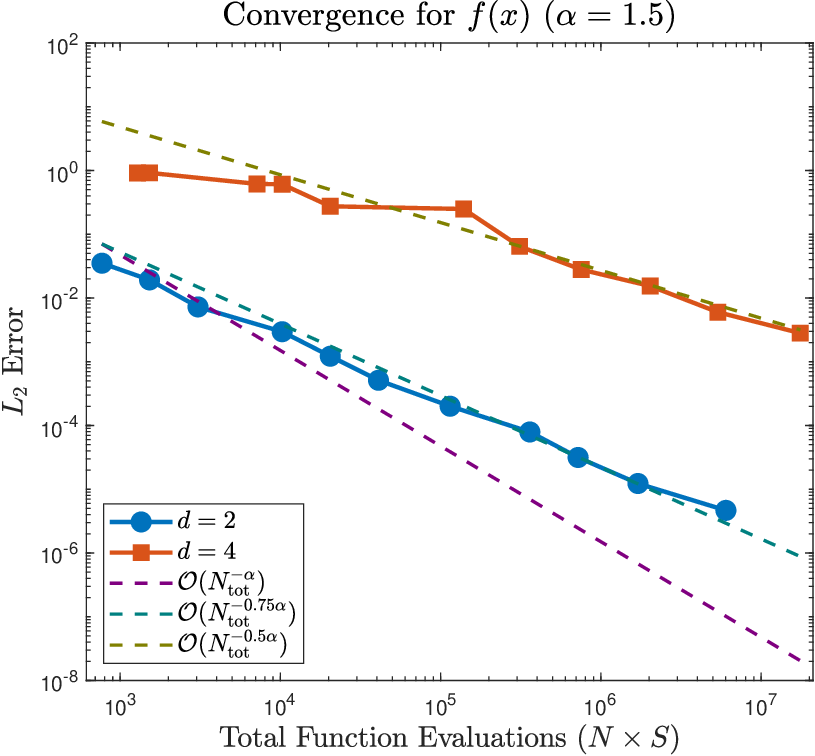}
        \caption{$L_2$ error of $f$ w.r.t total cost $N_{tot}$}
        \label{subfig:pde_f_L2_Ntot}
    \end{subfigure}
    
    \vspace{0.4cm} 
    
    \begin{subfigure}{0.48\textwidth}
        \centering
        \includegraphics[width=\textwidth]{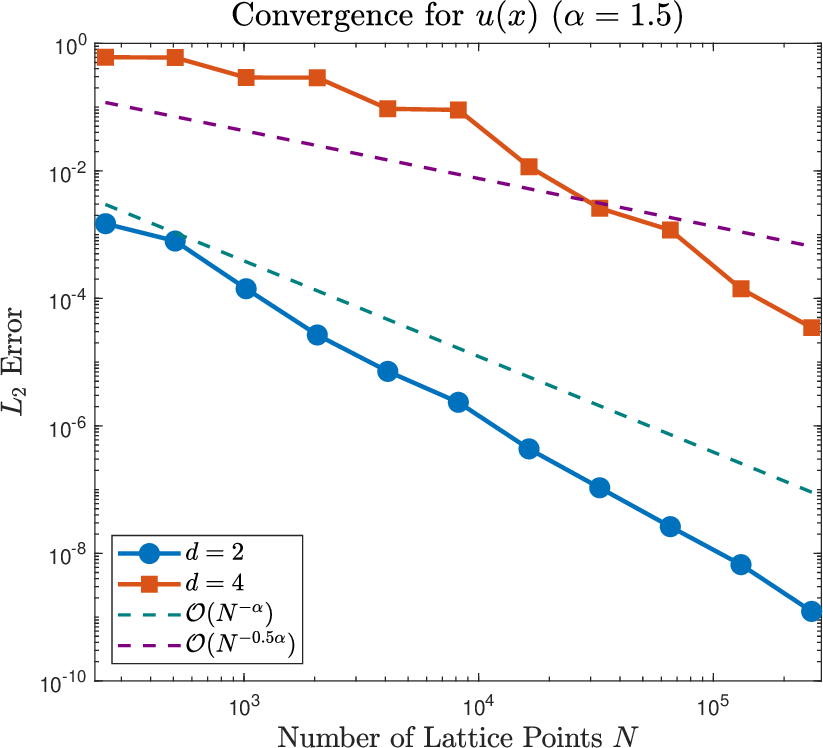}
        \caption{$L_2$ error of $u$ w.r.t base lattice size $N$}
        \label{subfig:pde_u_L2_N}
    \end{subfigure}
    \hfill
    \begin{subfigure}{0.48\textwidth}
        \centering
        \includegraphics[width=\textwidth]{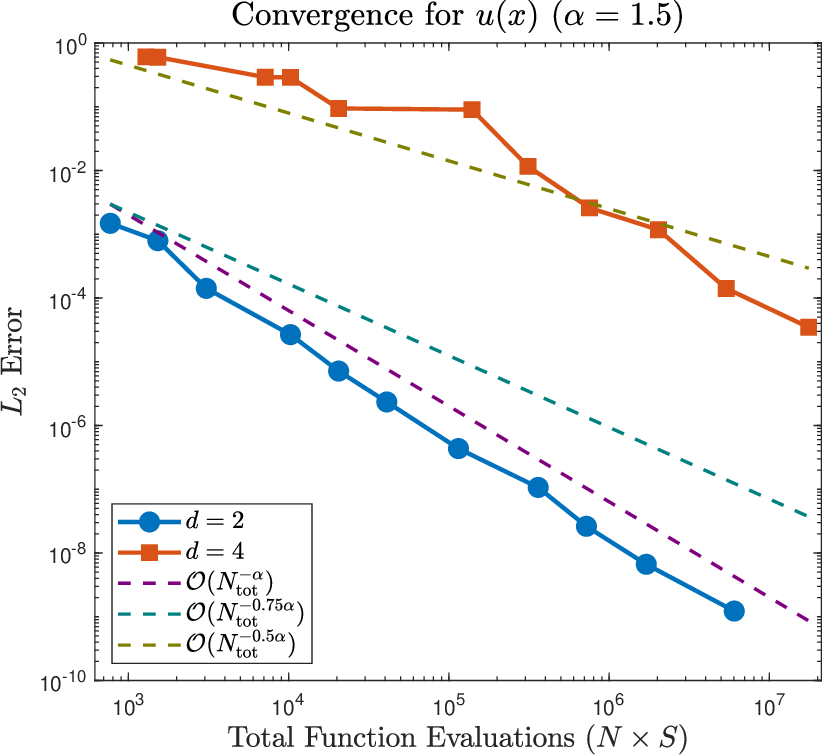}
        \caption{$L_2$ error of $u$ w.r.t total cost $N_{tot}$}
        \label{subfig:pde_u_L2_Ntot}
    \end{subfigure}
    
    \caption{The relative $L_2$ convergence of the proposed deterministic multiple-shift spectral solver
for the non-periodic Neumann problem $(\alpha = 1.5)$.}
    \label{fig:convergence_pde}
    \end{figure}

\section{Conclusion}
\label{sec:conclusion}
In this paper, we proposed a fully deterministic framework based on multiple-shift rank-1 lattice rules for multivariate function approximation without heavy pre-computation. By exploiting the algebraic structure of aliased frequency fibers, we introduced a uniform shift strategy. Coupled with an adaptive hybrid construction algorithm driven by the Chinese Remainder Theorem and the Weil bound, this approach structurally eliminates the reliance on probabilistic oversampling and massive pre-computed base lattices. We rigorously proved that this deterministic decoupling reduces the penalty factor associated with the maximum fiber length from $\mathcal{O}(R^{3/2})$ to $\mathcal{O}(R^{1/2})$ while maintaining the optimal convergence rate of $\mathcal{O}(N^{-\alpha+\tau})$ in the weighted Korobov space for any arbitrarily small $\tau > 0$.

Furthermore, we extended this deterministic framework to non-periodic domains via the tent transformation. By establishing a strict projection equivalence, we proved that the proposed algorithm attains the optimal approximation order in the half-period cosine space for both $L_2$ and $L_\infty$ errors, rigorously resolving a previously open theoretical problem regarding the exact decay rates in such spaces. Consequently, this mathematically validates the framework as a highly efficient meshless spectral solver for high-dimensional boundary value problems, preserving the source function's approximation rate for the Poisson equation with Neumann boundary conditions.

Numerical experiments extensively corroborate these theoretical guarantees. The adaptive strategy achieves an order-of-magnitude reduction in sampling cost compared to state-of-the-art probabilistic methods, while algebraically bounding the condition number to ensure absolute numerical stability in both function approximation and PDE solving.

Despite these advancements, the worst-case error bounds inherently retain a logarithmic dependence on the spatial dimension $d$, inherited from the maximum fiber length $R \le \mathcal{O}((\log N)^{d-1})$, which precludes strong tractability. Future work will explore analytical mechanisms to mitigate this dimension-dependent penalty, extend the deterministic framework to higher-order digital nets, and adapt the spectral solver for more complex physical systems with variable coefficients or non-rectangular geometries.


\bibliographystyle{plain}
\bibliography{references_ms}
\end{document}